\documentclass[a4paper]{article}
\usepackage{hyperref}
\usepackage{graphicx}
\usepackage{mathrsfs}
\usepackage{amsmath}
\usepackage{enumerate} 
\usepackage{amsxtra,amssymb,latexsym, amscd,amsthm}
\usepackage{indentfirst}
\usepackage{color}
\usepackage[utf8]{inputenc}
\usepackage[mathscr]{eucal}
\usepackage{amsfonts}
\usepackage{graphics}
\usepackage{multirow}
\usepackage{array}
\usepackage{subfigure}

\usepackage{graphicx}

\usepackage{cite}
\usepackage{wrapfig}



\usepackage{ushort}
\usepackage{pdflscape}
\usepackage{rotating}
\usepackage{pgfplots}

    {\clearpage\pagebreak[4]\global\pdfpageattr\expandafter{\the\pdfpageattr/Rotate 90}}%
    {\clearpage\pagebreak[4]\global\pdfpageattr\expandafter{\the\pdfpageattr/Rotate 0}}%

\usepackage{array}
\usepackage{booktabs}
\newcommand{\footremember}[2]{%
    \footnote{#2}
    \newcounter{#1}
    \setcounter{#1}{\value{footnote}}%
}

\def\R{{\mathbb R}}

\def\N{{\mathbb N}}
\def\S{{\mathbb S}}

\DeclareMathOperator{\conv}{conv}

\newcommand{\argmin}{\arg\min}

\newtheorem{theorem}{\bf Theorem}[section]
\newtheorem{lemma}[theorem]{\bf Lemma}
\newtheorem{remark}[theorem]{\bf Remark}
\newtheorem{assumption}[theorem]{\bf Assumption}

\providecommand{\keywords}[1]
{
  \small	
  \textbf{\textbf{Keywords:}} #1
}
\pgfplotsset{compat=1.13}

\begin{document}

\definecolor{qqzzff}{rgb}{0,0.6,1}
\definecolor{ududff}{rgb}{0.30196078431372547,0.30196078431372547,1}
\definecolor{xdxdff}{rgb}{0.49019607843137253,0.49019607843137253,1}
\definecolor{ffzzqq}{rgb}{1,0.6,0}
\definecolor{qqzzqq}{rgb}{0,0.6,0}
\definecolor{ffqqqq}{rgb}{1,0,0}
\definecolor{uuuuuu}{rgb}{0.26666666666666666,0.26666666666666666,0.26666666666666666}
\newcommand{\vi}[1]{\textcolor{blue}{#1}}
\newif\ifcomment
\commentfalse
\commenttrue
\newcommand{\comment}[3]{%
\ifcomment%
	{\color{#1}\bfseries\sffamily#3%
	}%
	\marginpar{\textcolor{#1}{\hspace{3em}\bfseries\sffamily #2}}%
	\else%
	\fi%
}
\newcommand{\victor}[1]{
	\comment{blue}{V}{#1}
}
\definecolor{oucrimsonred}{rgb}{0.6, 0.0, 0.0}
\newcommand{\jean}[1]{
	\comment{oucrimsonred}{J}{#1}
}
\definecolor{cadmiumgreen}{rgb}{0.0, 0.42, 0.24}
\newcommand{\hoang}[1]{
	\comment{cadmiumgreen}{H}{#1}
}

\title{Comparing different subgradient methods
 for solving convex optimization problems with functional constraints}
\author{%
Thi Lan Dinh\footremember{1}{Universit\"at G\"ottingen,\ Lotzestra{\ss}e 16-18, 37083 G\"ottingen, Germany.} \and
Ngoc Hoang Anh Mai\footremember{2}{University of Konstanz, Universit\"atsstra{\ss}e 10, D-78464 Konstanz, Germany.} 
  }
\maketitle
\begin{abstract}
\small
We consider the problem of minimizing a convex, nonsmooth function subject to a closed convex constraint domain. 
The methods that we propose are reforms of subgradient methods based on Metel--Takeda's paper [Optimization Letters 15.4 (2021): 1491-1504] and Boyd's works [Lecture  notes  of  EE364b, Stanford University, Spring 2013-14, pp. 1-39]. 
While the former has complexity $\mathcal{O}(\varepsilon^{-2r})$ for all $r> 1$,  the complexity of the latter is $\mathcal{O}(\varepsilon^{-2})$. 
We perform some comparisons between these two methods using several test examples.

\end{abstract}
\tableofcontents
\section{Introduction}

Given a convex function $f_0:\R^n\to\R$ and a closed convex domain $C\subset \R^n$, consider the convex optimization problem (COP):
\begin{equation}\label{eq:nonsmooth.problem}
    p^\star=\inf_{x\in C}\ f_0(x)\,,
\end{equation}
where $f_0$ might be non-differentiable.
It is well known that the
optimal value and an optimal solution for problem \eqref{eq:nonsmooth.problem} can be approximated as closely as desired by using subgradient methods. 
Although developed very early by Shor in the Soviet Union (see \cite{shor2012minimization}), subgradient methods are still highly competitive with state-of-the-art such as interior-point and Newton methods.
They can be applied to problem \eqref{eq:nonsmooth.problem} with a very large number of variables $n$ because of their little memory requirement.

We might classify problems of the form \eqref{eq:nonsmooth.problem} into three groups as follows:
\begin{enumerate}[(a)]
    \item Unconstrained case: $C=\R^n$.  
    It can be efficiently solved by various methods, in particular with
    Nesterov's method of weighted dual averages \cite{nesterov2009primal}. 
    His method has an optimal convergence rate $\mathcal{O}(\varepsilon^{-2})$.
    
    \item Inexpensive projection on $C$.
    In this case, the projected subgradient method introduced by Alber, Iusem, and Solodov \cite{alber1998projected} can be used for large-scale problems for which interior-point or Newton methods cannot be used.
    \item COP with functional constraints (also known as the standard form in the literature): 
    \begin{equation}
        C=\{x\in\R^n:f_i(x)\le 0\,,\,i=1,\dots,m\,,\,Ax=b\}\,,
    \end{equation}
    where $f_i$ is  a convex function on $\R^n$, for $i=1,\dots,m$, $A$ is a real matrix, and $b$ is a real vector.
    To solve this problem, Metel and Takeda suggest the subgradient method of weighted dual averages \cite{metel2021primal} based on Nesterov's idea.
    They also obtain the optimal convergence rate $\mathcal{O}(\varepsilon^{-2})$ for their method.
    Earlier, Boyd proposes a primal-dual subgradient method in his lecture notes \cite{boyd2014subgradient} without complexity analysis.
    We show that the complexity of Boyd's method is suboptimal.
\end{enumerate}
We are interested in the following case of large-scale COP \eqref{eq:nonsmooth.problem} in group (c):
\begin{assumption}
\label{ass:hard.cond}
    The objective and constraint functions F are non-differentiable at the optimal solution(s) of \eqref{eq:nonsmooth.problem}, non-strongly convex, and have expensive evaluations of their values and subgradients.
\end{assumption}
An example of this type is the nonsmooth form of a semidefinite program (see \eqref{eq:SDP.NS}).
It is hard to use the following methods to solve COP \eqref{eq:nonsmooth.problem} efficiently under Assumption \ref{ass:hard.cond}:
\begin{itemize}
    \item The methods only applied to problems with differentiable functions (e.g., sequential quadratic programming \cite{nocedal2006numerical}) frequently fail for COP \eqref{eq:nonsmooth.problem} because $p^\star$ is  attained at some $x^\star$ where $f_i$ is non-differentiable.
    \item Next, the expensive evaluation of  functions $f_i$ leads to bundle methods (e.g., \cite{monjezi2020filter}) might be less effective than subgradient methods. It is due to their sample size requires to be sufficiently large at each iteration.
    In contract, subgradient methods require an evaluation and a subgradient of $f_i$ at each iteration.
    \item Smoothing techniques can be applied for COP \eqref{eq:nonsmooth.problem}. However, it is also expensive for the evaluations to ensure smoothness (e.g., a smooth approximation for the maximal eigenvalue \cite{chen2004smooth}).
    
    \item Efficient methods exploiting the strong convexity (e.g., \cite{beck20141}) cannot be used because the objective and constraint functions of COP \eqref{eq:nonsmooth.problem} might not be strongly convex.
    \end{itemize}
Overcoming these issues, subgradient methods are suitable solvers for COP \eqref{eq:nonsmooth.problem} with Assumption \ref{ass:hard.cond}.

\paragraph{Contribution.}  Our contribution involving the previous group (c) is twofold:
\begin{enumerate}
    \item In Section \ref{sec:dsg}, we provide an alternative dual subgradient method with several dual variables based on Metel--Takeda's method   \cite{metel2021primal}. 
    Our dual subgradient method has the same complexity $\mathcal{O}(\varepsilon^{-2})$ a Metel--Takeda's.
    Moreover, it is much more efficient than Metel--Takeda's in practice, as shown in Section \ref{sec:experiments}.
    \item We provide in Section \ref{sec:pds} a primal-dual subgradient method based on the nonsmooth penalty method and Boyd's method in \cite[Section 8]{boyd2014subgradient}. 
    Our primal-dual subgradient method converges with complexity $\mathcal{O}(\varepsilon^{-2r})$, for all $r>1$.
\end{enumerate}
To compare, we recall Nesterov's subgradient method \cite[Section 3.2.4]{nesterov2018lectures} in Section \ref{sec:sg}. 
Our experiments in Section 6 show that Nesterov's  and Metel--Takeda's subgradient methods are less efficient (although they have optimal convergence rate $\mathcal{O}(\varepsilon^{-2})$).
In contrast, our dual subgradient method and primal-dual subgradient method are efficient in practice for COP with a large number of variables ($n\ge 1000$).

\section{Preliminaries}
Let $f:\R^n\to \R$ be a real-valued function on the Euclidean space $\R^n$.
We say that $f$ is convex if
\begin{equation}
    f(tx+(1-t)y)\le tf(x)+(1-t)f(y)\,,\quad\forall\  (x,y)\in \R^n\times\R^n\,,\,\forall\  t\in [0,1]\,.
\end{equation}
 If $f:\R^n\to \R$ is a convex function, the subdifferential at a point $\overline{x}$ in $\R^n$, denoted by $\partial f(\overline{x})$, is defined by the set
 \begin{equation}
     \partial f(\overline{x}):=\{g\in\R^n\,:\,f(x)-f(\overline{x})\ge g^\top (x-\overline{x})\}\,.
 \end{equation}
A vector $g$ in $\partial f(\overline{x})$ is called a subgradient at $\overline{x}$. 
The subdifferential is always a nonempty convex compact set.

From now on, we focus on solving 
constrained convex optimization problem with functional constraints:
\begin{equation}\label{eq:CCOP}
    p^\star=\inf_{x\in\R^n}\{f_0(x)\ :\ f_i(x)\le 0\,,\, i=1,\dots,m\,,\quad Ax= b\}\,,
\end{equation}
where $f_i:\R^n\to \R$ is a convex  function, for $i=0,1,\dots,m$, $ A\in \R^{l\times n}$ and $ b\in \R^{l}$.
Let $x^\star$ be an optimal solution for problem \eqref{eq:CCOP}, i.e., $f_i(x^\star)\le 0$, for $i=1,\dots,m$, $Ax^\star=b$ and $f_0(x^\star)=p^\star$.

\section{Subgradient method (SG)}
\label{sec:sg}
This section is devoted to recalling the simple subgradient method for \eqref{eq:CCOP}.
This method is introduced in \cite[Section 3.2.4]{nesterov2018lectures} due to Nesterov for convex problems with only inequality constraints.
We extend Nesterov's method to convex problems with both inequality and equality constraints in a trivial way.

It is easy to see that problem \eqref{eq:CCOP} is equivalent to problem:
\begin{equation}\label{eq:max.con.problem}
    p^\star=\inf_{x\in\R^n}\{f_0(x)\,:\,\overline{f}(x)\le 0\}\,,
\end{equation}
where 
\begin{equation}\label{eq:alter.const}
    \overline{f}(x)=\max\ \{\ f_1(x),\dots,f_m(x),|a_1^\top x-b_1|,\dots,|a_l^\top x-b_l|\ \}\,.
\end{equation}
Here $a_j^\top$ is the $j$th row of $A$, i.e., $A=\begin{bmatrix}
a_1^\top\\
\dots\\
a_l^\top\end{bmatrix}$. 
If $\overline x$ is an optimal solution for problem \eqref{eq:max.con.problem}, $\overline x$ is also an optimal solution for problem \eqref{eq:CCOP}.

Let 
$g_0(x)\in\partial f_0(x)$,
$\overline{g}(x)\in \partial \overline{f}(x)$ and consider the following method to solve problem \eqref{eq:max.con.problem}:
\begin{align}
\label{alg:sg}
\framebox{
\parbox{9cm}{
\begin{center}
\textbf{SG}
\end{center}
\small
{Initialization: $\varepsilon>0$, $x^{(0)}\in\R^n$.}\\
{For $k =0,1,\dots,K$ do:}
\begin{enumerate}
    \item If $\overline{f}(x^{(k)})\le \varepsilon$, then $x^{(k+1)}:=x^{(k)}-\frac{\varepsilon}{\|g_0(x^{(k)})\|_2^2}g_0(x^{(k)})$;
	\item Else, 
	$x^{(k+1)}:=x^{(k)}-\frac{\overline{f}(x^{(k)})}{\|\overline{g}(x^{(k)})\|_2^2}\overline{g}(x^{(k)})$.	
\end{enumerate}
}}
\end{align}

To guarantee the convergence for method \eqref{alg:sg}, we need the following assumption:
\begin{assumption}
	\label{ass:bound.subgrad}
	The norms of the subgradients of $f_0,f_1,\dots,f_m$ and the values of $f_1,\dots,f_m$ are bounded on any compact subsets of $\R^n$.
\end{assumption}

For every $\varepsilon>0$ and $K\in\N$, let $\mathcal{I}_\varepsilon(K)$ be the set of iterations and  $p^{(K)}_\varepsilon$ be a value  defined as follows:
\begin{equation}
    \mathcal{I}_\varepsilon(K):=\{k\in\{0,1,\dots,K\}:\overline{f}(x^{(k)})\le \varepsilon\}\qquad\text{and}\quad p^{(K)}_\varepsilon:=\min_{k\in \mathcal{I}_\varepsilon(K)}f_0(x^{(k)})\,.
\end{equation} 
The optimal worst-case performance guarantee of method \eqref{alg:sg} is stated in the following theorem:
\begin{theorem}\label{theo:convergence.sg}
Consider problem \eqref{eq:CCOP}.
Let Assumption \ref{ass:bound.subgrad} hold.
If the number of iterations $K$ in method \eqref{alg:sg} is big enough, $K\ge \mathcal{O}(\varepsilon^{-2})$,
then 
\begin{equation}
    \mathcal{I}_\varepsilon(K)\ne\emptyset\quad\text{ and }\quad p^{(K)}_\varepsilon\le p^\star + \varepsilon\,.
\end{equation}
\end{theorem}
The proof of Theorem \ref{theo:convergence.sg} is similar to the proof of \cite[Theorem 3.2.3]{nesterov2018lectures}.

\section{Dual subgradient method (DSG)} 
\label{sec:dsg}
This section provides an alternative dual subgradient method for \eqref{eq:CCOP}. 
Initially developed by Metel and Takeda in \cite{metel2021primal}, the dual subgradient method involves a single dual variable for problem \eqref{eq:max.con.problem} with the single constraint $\overline{f}(x)\le 0$. 
Our dual subgradient method is a generalization of Metel--Takeda's result and is based on the Lagrangian function of problem \eqref{eq:CCOP}  with several dual variables.

Consider the following Lagrangian function of problem \eqref{eq:CCOP}:
\begin{equation}
    L(x,\lambda,\nu)=f_0(x) + \lambda^\top F(x) +\nu^\top (Ax- b)\,,
\end{equation}
for $x\in\R^n$, $\lambda\in \R_+^m$ and $\nu \in\R^l$.
Here 
\begin{equation}\label{eq:def.max0}
    F(x):=(F_1(x),\dots,F_m(x))\text{ with }F_i(x)=\max\{f_i(x),0\}\,,\, i=1,\dots,m\,.
\end{equation}

The dual of problem \eqref{eq:CCOP} reads as:
\begin{equation}\label{eq:dual.CCOP}
    d^\star = \sup_{\lambda\in\R_+^m,\ \nu\in\R^l} \ \inf_{x\in\R^n} L(x,\lambda,\nu)\,.
\end{equation}
We make the following assumption:
\begin{assumption}\label{ass:strong.dual}
Strong duality holds for primal-dual \eqref{eq:CCOP}-\eqref{eq:dual.CCOP}, i.e., $p^\star=d^\star$ and \eqref{eq:dual.CCOP} has an optimal solution $(\lambda^\star,\nu^\star)$.
\end{assumption}
It implies that $(x^\star,\lambda^\star,\nu^\star)$ is a saddle-point of the Lagrangian function $L$, i.e.,
\begin{equation}\label{eq:saddle.point}
    L(x^\star, \lambda, \nu)\le L(x^\star, \lambda^\star,\nu^\star)=p^\star\le L(x,\lambda^\star,\nu^\star)\,,
\end{equation}
for all $x\in\R^n$, $\lambda\in \R_+^m$ and $\nu \in\R^l$.

Given $C$ as a subset of $\R^n$, denote by $\conv(C)$  the convex hull  generated by $C$, i.e., $\conv(C)=\{\sum_{i=1}^r{t_ia_i}:a_i\in C\,,\,t_i\in[0,1]\,,\,\sum_{i=1}^rt_i=1\}$.

With $z=(x,\lambda,\nu)$, we use the following notation:
\begin{itemize}
    \item $g_0(x)\in\partial f_0(x)$;
\item $g_i(x)\in\partial F_i(x)=\begin{cases}\partial f_i(x)&\text{ if }f_i(x)>0\,,\\
    \text{conv}(\partial f_i(x)\cup \{0\})&\text{ if }f_i(x)=0\,,\\
    \{0\}&\text{ otherwise};
    \end{cases}$
\item $G_x(z)\in\partial_x L(z)=\partial f_0(x)+\sum_{i=1}^m\lambda_i\partial F_i(x)\nonumber+A^\top\nu$;
\item $G_{\lambda}(z)=\nabla_{\lambda} 
L(z)=F(x)$ and  $G_{\nu}(z)=\nabla_{\nu} 
L(z)=Ax-b$;
\item $G(z)=(G_{x}(z),-G_{\lambda}(z),-G_{\nu}(z))$.
\end{itemize}
Setting $z^{(k)}:=(x^{(k)},\lambda^{(k)},\nu^{(k)})$, we consider the following method:
\begin{align}
\label{alg:dsg}
\framebox{
\parbox{8cm}{
\begin{center}
\textbf{DSG}
\end{center}
\small
{Initialization: $z^{(0)}=(x^{(0)},\lambda^{(0)},\nu^{(0)})\in \R^n\times \R^m_+\times \R^l$; 
		$s^{(0)}=0$; $\hat{x}^{(0)}=0$; $\delta_0=0$; $\beta_0=1$.}\\
{For $k =0,1,\dots,K$ do:}
\begin{enumerate}
    \item $s^{(k+1)}=s^{(k)}+\frac{G(z^{(k)})}{\|G(z^{(k)})\|_2}$;
	\item $z^{(k+1)}=z^{(0)}-\frac{s^{(k+1)}}{\beta_{k}}$;	
	\item $\beta_{k+1}=\beta_k+\frac{1}{\beta_k}$;
	\item $\delta_{k+1}=\delta_k+\frac{1}{\|G(z^{(k)})\|_2}$;
	\item $\hat{x}^{(k+1)}=\hat{x}^{(k)}+\frac{x^{(k)}}{\|G(z^{(k)})\|_2}$;
	\item $\overline{x}^{(k+1)}=\delta_{k+1}^{-1}\hat{x}^{(k+1)}$.
\end{enumerate}
}}
\end{align}
Since $G_{\lambda}(z^{(k)})=F(x^{(k)})\ge 0$, it holds that $\lambda^{(k+1)}_i\ge \lambda^{(k)}_i\ge \lambda^{(0)}_i\ge 0$, for $i=1,\dots,m$.

Under Assumption \ref{ass:bound.subgrad}, we  obtaint the convergence rate of order $\mathcal{O}(K^{-1/2})$ in the following theorem:
\begin{theorem}\label{theo:convergence.dsg}
Consider problem \eqref{eq:CCOP}.
Let Assumptions \ref{ass:bound.subgrad} and \ref{ass:strong.dual} hold.
Let $\varepsilon>0$.
If the number of iterations $K$ in method \eqref{alg:dsg} is big enough, $K\ge \mathcal{O}(\varepsilon^{-2})$,
then 
\begin{equation}
    \|F(\overline{x}^{(K+1)})\|_2+\|A\overline{x}^{(K+1)}-b\|_2\le \varepsilon\text{ and }\quad f_0(\overline{x}^{(K+1)})\le p^\star + \varepsilon\,.
\end{equation}
\end{theorem}

The proof of Theorem \ref{theo:convergence.dsg}, based on Lemma  \ref{convergence}, proceeds the same as the proofs in \cite{metel2021primal}.

\section{Primal-dual subgradient method (PDS)}
\label{sec:pds}
This section extends the primal-dual subgradient method introduced in Boyd's lecture notes \cite[Section 8]{boyd2014subgradient}.
The idea of our method is to replace the augmentation of the augmented Lagrangian considered in \cite[Section 8]{boyd2014subgradient} with a more general penalty term.
We also prove the convergence guarantee and provide a convergence rate order for this method.

Let $s\in[1,2]$ and $\rho>0$ be fixed.
With $F(x)$ defined as in \eqref{eq:def.max0}, consider an equivalent problem of \eqref{eq:CCOP}:
\begin{equation}\label{eq:modified.augmented.CCOP}
\begin{array}{rl}
    p^\star=\inf\limits_{x\in\R^n}&f_0(x) + \rho(\|F(x)\|_2^s+\| Ax- b\|_2^s)\\
    \text{s.t.}&  F_i(x)\le 0\,,\,i=1,\dots,m\,,\,  Ax= b\,.
    \end{array}
\end{equation}
Since $x^\star$ is an optimal solution for problem \eqref{eq:CCOP},  $x^\star$ is also an optimal solution for problem \eqref{eq:modified.augmented.CCOP}.
\begin{remark}\label{re:diff.augmentation}
Instead of using the augmentation with the  square of $l_2$-norm  $\|\cdot\|_2^2$ (according to the definition of the augmented problem in \cite[Section 8]{boyd2014subgradient}) we use the  additional penalty term with $l_2$-norm to the $s$th power $\|\cdot\|_2^s$ in problem \eqref{eq:modified.augmented.CCOP}.
This strategy has been discussed in \cite[pp. 513]{nocedal2006numerical} for the case of $s=1$.
\end{remark}

The Lagrangian of problem \eqref{eq:modified.augmented.CCOP} is of form:
\begin{equation}
    L_\rho(x,\lambda,\nu)=f_0(x) + \lambda^\top  F(x) +\nu^\top ( Ax- b)+ \rho(\| F(x)\|_2^s+\| Ax- b\|_2^s)\,,
\end{equation}
for $x\in\R^n$, $\lambda\in \R_+^m$ and $\nu \in\R^l$.

Let us define a set-valued mapping $T_\rho:\R^n\times\R^m_+\times \R^l\to  2^{\R^n\times\R^m_+\times \R^l}$ by
\begin{equation}
\begin{array}{rl}
    T_\rho(x,\lambda,\nu)=&
    \partial_x L_\rho(x,\lambda,\nu)\times
    (-\partial_\lambda L_\rho(x,\lambda,\nu))\times
    (-\partial_\nu L_\rho(x,\lambda,\nu))\\
    =&\partial_x L_\rho(x,\lambda,\nu)\times
     \{-F(x)\}\times
    \{b-Ax\}\,,
    \end{array}
\end{equation}
where $\partial_x L_\rho(x,\lambda,\nu)=\partial f_0(x)+\sum_{i=1}^m\lambda_i \partial  F_i(x)  + A^\top \nu +\rho\partial  \|F_i(\cdot)\|_2^s(x) +  \rho\partial  \|A\cdot-b\|_2^s(x)$.

The explicit formulas of the subgradients in the subdifferentials $\partial  \|F_i(\cdot)\|_2^s(x)$ and $\partial  \|A\cdot-b\|_2^s(x)$ are provided in Appendix \ref{app:subgrad.comp}.

We do the simple iteration:
\begin{equation}\label{eq:sheme}
    z^{(k+1)}=z^{(k)} - \alpha_k T^{(k)}\,,
\end{equation}
where $z^{(k)}=(x^{(k)},\lambda^{(k)},\nu^{(k)})$ is the $k$th iterate of the primal and dual variables, $T^{(k)}$ is any element of $T_{\rho}(z^{(k)})$, $\alpha_k > 0$ is the $k$th step size.

By expanding \eqref{eq:sheme} out, we can also write the method as:
\begin{align}
\label{alg:PDS}
\framebox{
\parbox{11cm}{
\begin{center}
\textbf{PDS}
\end{center}
\small
{Initialization: $x^{(0)}\in\R^n$, $\lambda^{(0)}\in\R^m_+$, $\nu^{(0)}\in\R^l$ and $(\alpha_k)_{k\in\N}\subset\R_+$.}\\
{For $k=0,1,\dots,K$ do:}
\begin{enumerate}
\item $\varrho^{(k)}=\begin{cases}
 s \| F^{(k)}\|_2^{s-2} F^{(k)}&\text{ if }F^{(k)}\ne 0\,,\\
0&\text{ otherwise;}
\end{cases}$
\item $\varsigma^{(k)}=\begin{cases}
 s\|Ax^{(k)}-b\|_2^{s-2}(Ax^{(k)}-b)&\text{ if }Ax^{(k)}\ne b\\
0 &\text{ otherwise;}
\end{cases}
$
    \item $x^{(k+1)}=x^{(k)} - \alpha_k [g_0^{(k)}+\sum_{i=1}^m(\lambda_i^{(k)}+ \rho\varrho^{(k)}_i) g_i^{(k)}+A^\top (\nu^{(k)} +  \rho \varsigma^{(k)})]$;
    \item $\lambda^{(k+1)}=\lambda^{(k)} +\alpha_k F^{(k)}\text{ and }\nu^{(k+1)}=\nu^{(k)} +\alpha_k (Ax^{(k)}-b)$.
\end{enumerate}
}}
\end{align}
Here we note:
\begin{itemize}
    \item $g_0^{(k)}\in \partial f_0 (x^{(k)})$, $g_i^{(k)}\in \partial F_i(x^{(k)})$ and $F_i^{(k)}:=F_i(x^{(k)})$, $i=1,\dots,m$.
    \item $T^{(k)}=\begin{bmatrix}
    g_0^{(k)}+\sum_{i=1}^m(\lambda_i^{(k)}+ \rho\varrho^{(k)}_i) g_i^{(k)}+A^\top (\nu^{(k)} +  \rho \varsigma^{(k)})\\
    -F^{(k)}\\
    b-Ax^{(k)}
    \end{bmatrix}
    $.
\end{itemize}
Note that $\lambda^{(k)}\ge 0$ since $F_i^{(k)}\ge 0$, $i=1,\dots,m$.
The case of $s=2$ is the standard primal-dual subgradient method in \cite[Section 8]{boyd2014subgradient}.

For every $\varepsilon>0$ and $K\in\N$, let
\begin{equation}
    \mathcal{I}_\varepsilon(K):=\{k\in\{0,1,\dots,K\}\,:\,\|F(x^{(k)})\|_2+\|Ax^{(k)}-b\|_2\le \varepsilon\}
\end{equation} and 
\begin{equation}
    p^{(K)}_\varepsilon:=\min_{k\in \mathcal{I}_\varepsilon(K)}f_0(x^{(k)})\,.
\end{equation}

We state that the method \eqref{alg:PDS} converges in the following theorem:
\begin{theorem}\label{theo:convergence}
Consider problem \eqref{eq:CCOP}.
Let Assumptions \ref{ass:bound.subgrad} and \ref{ass:strong.dual} hold.
Let $\delta\in(0,1)$.
Let the step size rule
\begin{equation}
    \alpha_k=\frac{\gamma_k}{\|T^{(k)}\|_2}\quad\text{ with }\quad\gamma_k=(k+1)^{-1+\delta/2}\,.
\end{equation}
Let $\varepsilon>0$.
If the number of iterations $K$ in method \eqref{alg:PDS} is big enough, $K\ge \mathcal{O}(\varepsilon^{-2s/\delta})$,
then 
\begin{equation}
    \mathcal{I}_\varepsilon(K)\ne\emptyset\quad\text{ and }\quad p^{(K)}_\varepsilon\le p^\star + \varepsilon\,.
\end{equation}
\end{theorem}
The proof of Theorem \ref{theo:convergence} is based on Lemma \ref{lem:theo:convergence}.
This proof is similar to the convergence proof of the (standard) primal-dual subgradient method in \cite[Section 8]{boyd2014subgradient}.
A mathematical mistake in the proof in \cite[Section 8]{boyd2014subgradient} is corrected.



\section{Numerical experiments}
\label{sec:experiments}
In this section, we report the results of numerical experiments obtained by solving convex optimization problems (COPs) with functional constraints. 
The experiments are performed in Python 3.9.1. 
The implementation of methods \eqref{alg:sg}, \eqref{alg:dsg} and \eqref{alg:PDS} is available online via the link:
\begin{center}
    \href{https://github.com/dinhthilan/COP}{{\bf https://github.com/dinhthilan/COP}}.
\end{center}

We use a desktop computer with an Intel(R) Pentium(R) CPU N4200 @ 1.10GHz and 4.00 GB of RAM. 
The notation for the numerical results is given in Table \ref{tab:nontation}.
\begin{table}
    \caption{\small Notation}
    \label{tab:nontation}
\small
\begin{center}
\begin{tabular}{|m{1.5cm}|m{9cm}|}
\hline
$n$&the number of variables of the COP\\
\hline
$m$&the number of inequality constraints of the COP\\
\hline
$l$&the number of equality constraints of the COP\\
\hline
SG &the COP 
solved by the subgradient method \eqref{alg:sg}\\
\hline
SingleDSG &the COP 
solved by the dual subgradient method with single dual variable \cite[Algorithm 1]{metel2021primal}\\
\hline
MultiDSG &the COP 
solved by the dual subgradient method with multi-dual-variables \eqref{alg:dsg}\\
\hline
PDS &the COP 
solved by the primal-dual subgradient method \eqref{alg:PDS}\\
\hline
$s$ & the power of $l_2$-norm in the additional penalty term for PDS \\
\hline
$\rho$ & the penalty coefficients  for PDS\\
\hline
val& the approximate optimal value for the COP\\
\hline
val$^\star$& the exact optimal value for the COP\\
\hline
gap& the relative optimality gap w.r.t. the exact value val$^\star$, i.e.,

$\text{gap}=|\text{val}-\text{val}^\star|/{(1+\max\{|\text{val}^\star|,|\text{val}|\})}$\\
\hline 
infeas & the infeasibility of the approximate optimal solution\\
\hline
time & the running time in seconds\\
\hline
$\varepsilon$ & the desired accuracy of the approximate solution \\
\hline
$K$ & the number of iterations \\
\hline

\end{tabular}    
\end{center}
\end{table}
The value and the infeasibility of SG, SingleDSG, MultiDSG, and PDS at the $k$th iteration are computed as in Table \ref{tab:val.infeas}.
\begin{table}
    \caption{\small The value and the infeasitility at the $k$th iteration.}
    \label{tab:val.infeas}
\footnotesize
\begin{center}
   \begin{tabular}{|c|c|c|c|}
        \hline
        Method& Complexity&
        val&
        infeas\\
\hline
SG& $\mathcal{O}(\varepsilon^{-2})$ &$f_0(x^{(k)})$ & $\max\{\overline{f}(x^{(k)}),0\}$\\
\hline
SingleDSG& $\mathcal{O}(\varepsilon^{-2})$ &$f_0(\overline{x}^{(k)})$& $\max\{\overline{f}(\overline{x}^{(k)}),0\}$\\
\hline
MultiDSG& $\mathcal{O}(\varepsilon^{-2})$ &$f_0(\overline{x}^{(k)})$ &$\|F(\overline{x}^{(k)})\|_2+\|A\overline{x}^{(k)}-b\|_2$\\
\hline
PDS&$\mathcal{O}(\varepsilon^{-2r})\,,\,\forall\ r>1$& $f_0(x^{(k)})$& $\|F({x}^{(k)})\|_2+\|A{x}^{(k)}-b\|_2$  \\
\hline
\end{tabular}    
\end{center}
\end{table}

\subsection{Randomly generated test problems}
We construct randomly generated test problems in  form:
\begin{equation}\label{eq:def.random}
    \min_{x\in\R^n}\{ c^{\top}x\ :\ x \in \Omega\,,\,Ax=b\}\,,
\end{equation}
where $c\in\R^n$, $A\in\R^{l\times n}$, $b\in\R^l$, and $\Omega$ is a convex domain such that:
\begin{itemize}
    \item  Every entry of $c$ and $A$ is taken in $[-1,1]$ with uniform distribution.
    \item The domain $\Omega$ is chosen in the following two cases:
    \begin{itemize}
        \item Case 1: $\Omega:=\{x\in\R^n:\|x\|_1\le 1\}$.
        \item Case 2: $\Omega:=\{x\in[-1,1]^n: \max\{-\log(x_1+1),x_2\}\le 1\}$.
    \end{itemize}
    \item With a random point $\overline x$ in $\Omega$, we take $b:=A\overline x$.
\end{itemize}
Let us apply SG, SingleDSG, multiDSG, and PDS to solve problem \eqref{eq:def.random}.
The size of the test problems and the setting of our software are given in Table \ref{tab:random.prob}.
\begin{table}
\caption{\small Randomly generated test problems.}
    \label{tab:random.prob}
    {\small\begin{itemize}
        \item Setting: $l=\lceil n/4 \rceil$ $\varepsilon=10^{-3}$, $\rho=1/s$.
    \end{itemize}}
\small
\begin{center}
   \begin{tabular}{|c|c|c|c|c|c|c|}
        \hline
        \multirow{2}{*}{Id}&
        \multirow{2}{*}{Case}&
        \multirow{2}{*}{$\delta$}&
        \multirow{2}{*}{$K$}& 
        \multicolumn{3}{c|}{Size}\\
\cline{5-7}
  & & &  &$n$ & $m$ & $l$\\
\hline  
1&1&0.5&$10^4$ & 10 &1 & 2 \\ \hline
2&1& 0.5& $10^5$ &100& 1 &15 \\  \hline
3&1&0.99&$10^6$ &1000 & 1& 143\\ \hline
4&2&0.5& $10^4$&  10 & 21 & 2 \\  \hline
5&2&0.5&$10^5$ &$100$&201&15\\  \hline
6&2 &0.99&$10^5$&$1000$ &2001& 143 \\
\hline
\end{tabular}    
\end{center}
\end{table}
The numerical results are displayed in Table \ref{tab:random.result}.
\begin{table}
\caption{\small Numerical results  for randomly generated test problems}
    \label{tab:random.result}
\small
\begin{center}
   \begin{tabular}{|c|c|c|c|c|c|c|c|c|c|}
        \hline
        \multirow{2}{*}{Id}&
        \multicolumn{3}{c|}{SG} & \multicolumn{3}{c|}{SingleDSG}&      \multicolumn{3}{c|}{MultiDSG} \\ \cline{2-10}
  & val& infeas &time&
val&  infeas &time &
val & infeas & time\\
\hline  
1 & -0.6363 & 0.2584 & 1 & -0.9609 & 0.0806 & 1 &-0.8385 & 0.0140&1\\ 
\hline
2& 0.2711& 0.2389&19 &-0.9672&0.1832&35&-0.9003&0.0177&19\\ \hline
3& -0.0032 & 0.0849 &2356 & -0.9938 & 0.0748 & 4704 & -0.8649 & 0.0153 & 1380 \\ \hline
4& -2.2255 &2.3955 & 1& -5.7702 & 2.0534 & 2 & -4.1255 & 0.0056& 5 \\ \hline
5&7.6548&7.7407&432&-141.43&7.0433&84&-37.712& 0.0097&575\\ \hline
6&-1.4176&38.806&1079&-316.62&35.684&2382&-401.07&0.0389&8059\\
        \hline
        \hline
        \multirow{2}{*}{Id}&
        \multicolumn{3}{c|}{PDS with $s=1$} & \multicolumn{3}{c|}{PDS with $s=1.5$ }&      \multicolumn{3}{c|}{PDS with $s=2$} \\ \cline{2-10}
  & val& infeas &time&
val&  infeas &time &
val & infeas & time\\
\hline  
1 & -0.8213 & 0.0007 & 1 & -0.8226 & 0.0034 & 1 &-0.8209 & 0.0003&1\\ \hline
2& -0.8830& 0.0008&20&-0.9004&0.0957&20&-0.8913&0.0262&20\\ \hline
3& -0.8550 & 0.0241 &2434 & -0.8428 & 0.0043 & 2673 & -0.8430 & 0.0046 & 2424 \\ \hline
4& -4.0708 &0.0025 & 2& -4.1017 & 0.0029 & 2 & -4.1037 & 0.0041& 2 \\ \hline
5&-37.022&0.0009&404&-37.453&0.0214&416&-37.506&0.0451&413\\ \hline
6&-398.71&0.0435&4802&-399.63&0.0475&4889&-399.76&0.0323&4602\\ 
\hline
\end{tabular}    
\end{center}
\end{table}
The convergences of SG, SingleDSG, MultiDSG, and PDS with $s\in\{1,1.5,2\}$ are illustrated in Figures \ref{fig:case1.n10}, \ref{fig:case1.n100}, \ref{fig:case1.n1000} for Case 1 and Figures \ref{fig:case2.n10}, \ref{fig:case2.n100}, \ref{fig:case2.n1000} for Case 2.
\begin{center}
    \begin{figure}
  {\includegraphics[width=.5\linewidth]{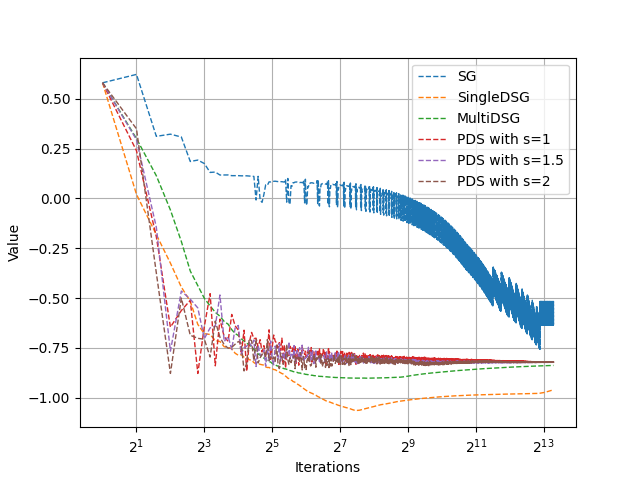}}\hfill
  {\includegraphics[width=.5\linewidth]{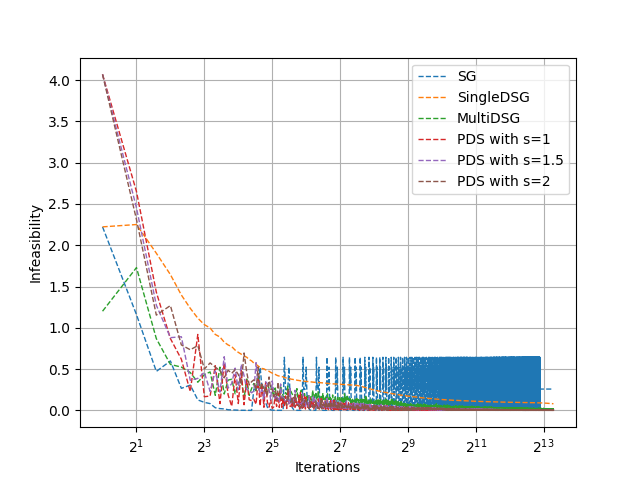}}\hfill
\caption{Illustration for Case 1 with $n=10$ (Id 1).}
\label{fig:case1.n10}
  {\includegraphics[width=.5\linewidth]{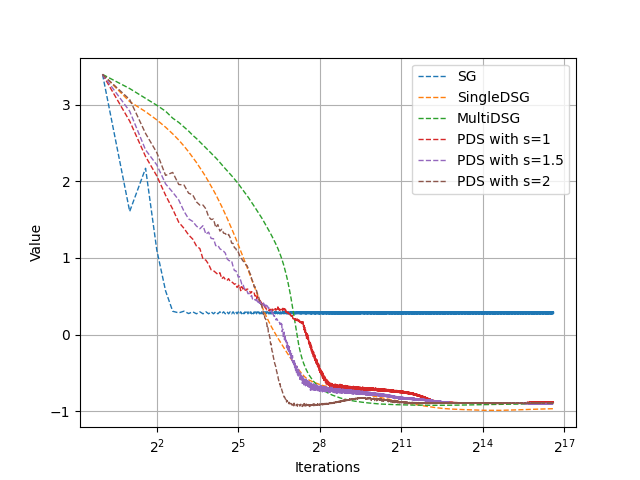}}\hfill
  {\includegraphics[width=.5\linewidth]{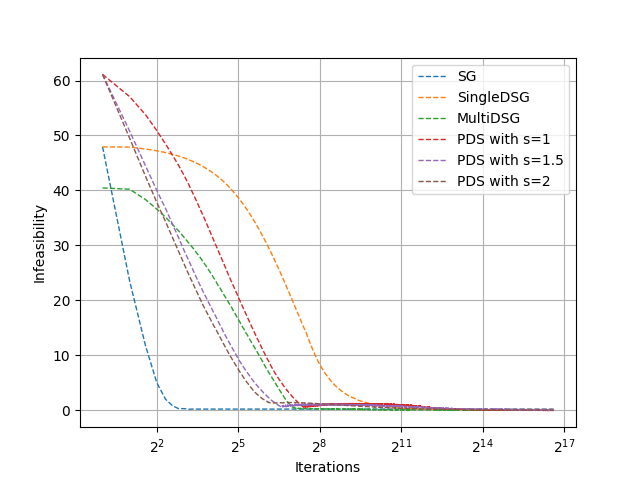}}\hfill
\caption{Illustration for Case 1 with $n=100$ (Id 2).}
\label{fig:case1.n100}
  {\includegraphics[width=.5\linewidth]{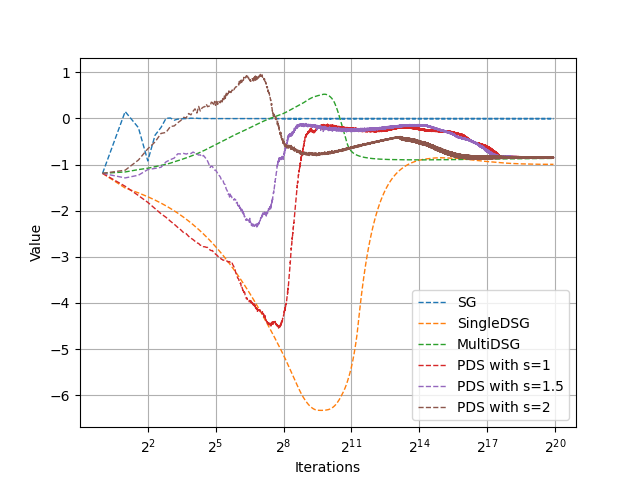}}\hfill
  {\includegraphics[width=.5\linewidth]{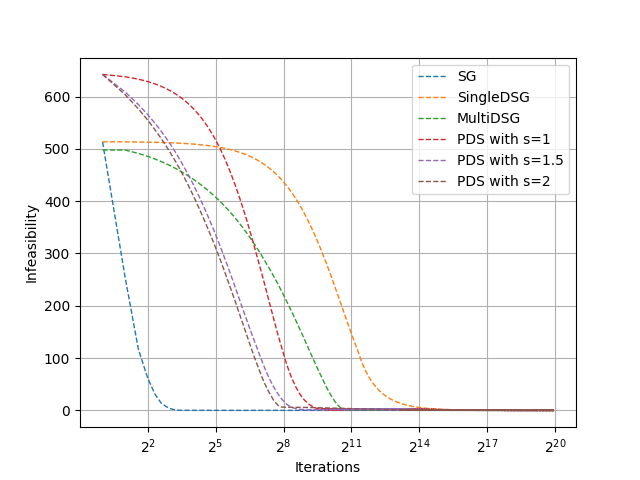}}\hfill
\caption{Illustration for Case 1 with $n=1000$ (Id 3).}
\label{fig:case1.n1000}
\end{figure}
\end{center}
\begin{center}
    \begin{figure}
  {\includegraphics[width=.5\linewidth]{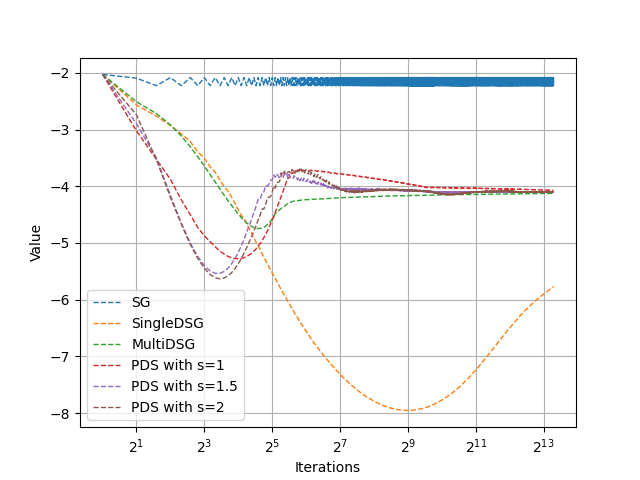}}\hfill
  {\includegraphics[width=.5\linewidth]{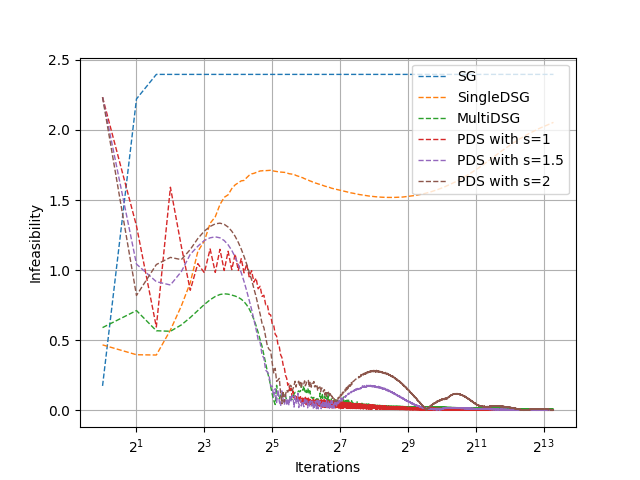}}\hfill
\caption{Illustration for Case 2 with $n=10$ (Id 4).}
\label{fig:case2.n10}
  {\includegraphics[width=.5\linewidth]{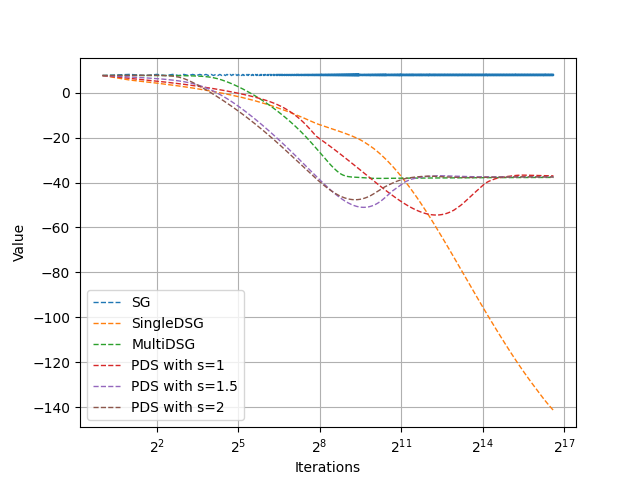}}\hfill
  {\includegraphics[width=.5\linewidth]{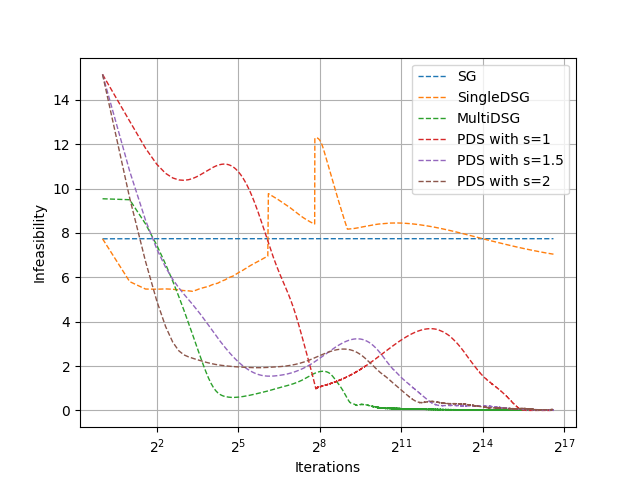}}\hfill
\caption{Illustration for Case 2 with $n=100$ (Id 5).}
\label{fig:case2.n100}
  {\includegraphics[width=.5\linewidth]{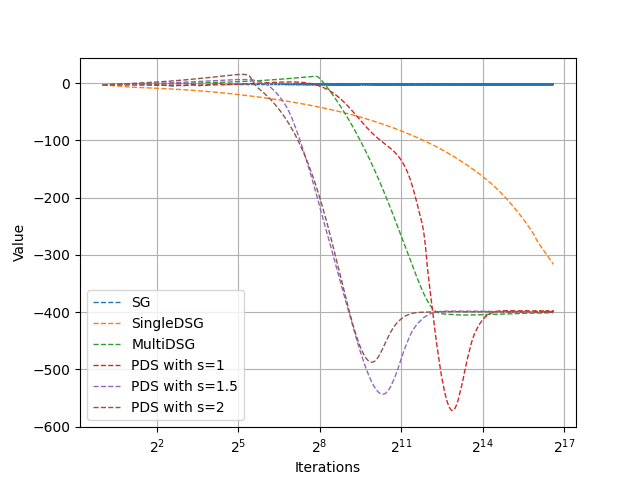}}\hfill
  {\includegraphics[width=.5\linewidth]{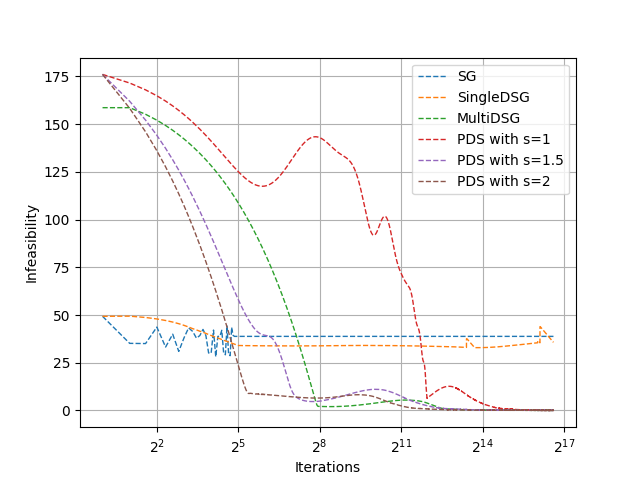}}\hfill
\caption{Illustration for Case 2 with $n=1000$ (Id 6).}
\label{fig:case2.n1000}
\end{figure}
\end{center}
These figures show that:
\begin{itemize}
    \item In Case 1, the values returned by SingleDSG, MultiDSG, and PDS with $s\in\{1,1.5,2\}$ have the same limit when the number of iterations increases except Id 1. 
In Id 1, the values returned by SingleDSG have a limit, which is not the same as one of the values returned by the other.
It is worth pointing out that the value of PDS with $s=2$ has the fastest convergence rate.
In this case, the infeasibilities of all methods converge to zero, and the infeasibility of SG has the fastest convergence rate.
\item In Case 2, the values returned by  MultiDSG and PDS with $s\in\{1,1.5,2\}$ have the same limit, while the values provided by SG and SingleDSG converge to different limits when the number of iterations increases. 
Moreover, the value of MultiDSG has the fastest convergence rate.
We observe similar behavior when considering the infeasibilities of these methods. 
\end{itemize}

For timing comparison from Table \ref{tab:random.result}, SingleDSG is the slowest in Case 1 while MultiDSG is the fastest one in this case.
Nevertheless, MultiDSG takes the most consuming time in Case 2.

\subsection{Linearly inequality constrained minimax problems}

We solve test problems MAD8, Wong2, and Wong3 from \cite[Table 4.1]{lukvsan2000test} using SG, SingleDSG, multiDSG, and PDS.

The size of the test problems and the setting of our software are given in Table \ref{tab:prob.minimax}.
\begin{table}
    \caption{\small Linearly inequality constrained minimax test problems.}
    \label{tab:prob.minimax}
    {\small\begin{itemize}
        \item Setting:  $l=0$, $\varepsilon=10^{-3}$, $\delta=0.5$, $\rho=1/s$, $K=10^5$.
    \end{itemize}}
\footnotesize
\begin{center}
   \begin{tabular}{|c|c|c|c|c|}
        \hline
        \multirow{2}{*}{Id}&
        \multirow{2}{*}{Problem}&
        \multirow{2}{*}{val$^\star$}&
        \multicolumn{2}{c|}{Size} \\ \cline{4-5}
& & &$n$ &$m$\\
\hline
7&MAD8& 0.5069 & 20 & 10  \\
\hline
8&Wong2& 24.3062 & 10 & 3  \\
\hline
9&Wong3& -37.9732 & 20 & 4 \\
\hline
\end{tabular}    
\end{center}
\end{table}
The numerical results are displayed in Table \ref{tab:minimax}.
\begin{table}
\caption{\small Numerical results for linearly inequality constrained minimax problems}
    \label{tab:minimax}
\footnotesize
\begin{center}
   \begin{tabular}{|c|c|c|c|c|c|c|}
        \hline
      \multirow{2}{*}{Id}& \multicolumn{2}{c|}{SG}&      \multicolumn{2}{c|}{SingleDSG } &      \multicolumn{2}{c|}{MultiDSG} \\ \cline{2-7}
 &val&  infeas&
val & infeas& val & infeas \\
\hline
7  &0.5065&0.0006 &0.4629 &  0.0325  & 0.5037 & 0.0038  \\
\hline
8  & 653.00  & 0.0000 &22.352 &0.6159  & 23.964 & 0.1017 \\
\hline
9 & 198.03  & 0.0000 &-39.638 & 0.9788 & -38.697 & 0.2851 \\
        \hline
        \hline
      \multirow{2}{*}{Id}& \multicolumn{2}{c|}{PDS with $s=1$ }&      \multicolumn{2}{c|}{PDS with $s=1.5$} &      \multicolumn{2}{c|}{PDS with $s=2$} \\ \cline{2-7}
 &val&  infeas&
val & infeas& val & infeas \\
\hline
7  & 0.5073 & 0.0000 & 0.5070 &0.0000& 0.5071&0.0000 \\
\hline
8  & 24.305& 0.0013 & 24.127 &0.0975 &  24.003 &0.1360\\
\hline
9 & -37.969& 0.0009  &-38.436 & 0.6214 & -38.628 & 0.8841\\
\hline
   \hline
      \multirow{2}{*}{Id}& \multicolumn{2}{c|}{SG}&      \multicolumn{2}{c|}{SingleDSG } &      \multicolumn{2}{c|}{MultiDSG} \\ \cline{2-7}
 & gap&time &
 gap& time& gap& time\\
\hline
7  &0.03\%  & 204 &2.92\% & 203  & 0.21\% & 221  \\
\hline
8  & 96.1\%  & 48 &7.72\% & 54  & 1.35\% & 57 \\
\hline
9 & 118\% & 76 &4.10\% & 86 & 1.82\% & 91 \\
        \hline
        \hline
      \multirow{2}{*}{Id}& \multicolumn{2}{c|}{PDS with $s=1$ }&      \multicolumn{2}{c|}{PDS with $s=1.5$} &      \multicolumn{2}{c|}{PDS with $s=2$} \\ \cline{2-7}
 &
 gap&time &
 gap& time& gap& time\\
\hline
7  & 0.01\% &113  & 0.01\% &116 &0.01\%& 113 \\
\hline
8  & 0.00\%& 29 & 0.71\% &21 &  1.20\% &27\\
\hline
9 & 0.01\%& 45  &1.18\% & 48 & 1.65\% & 43\\
\hline
\end{tabular}    
\end{center}
\end{table}
The convergences of SG, SingleDSG, MultiDSG, and PDS with $s\in\{1,1.5,2\}$ are plotted in Figures \ref{fig:lad.n10}, \ref{fig:lad.n100}, and \ref{fig:lad.n1000} for $\overline n\in\{10,100,1000\}$, respectively.
\begin{center}
    \begin{figure}
  {\includegraphics[width=.5\linewidth]{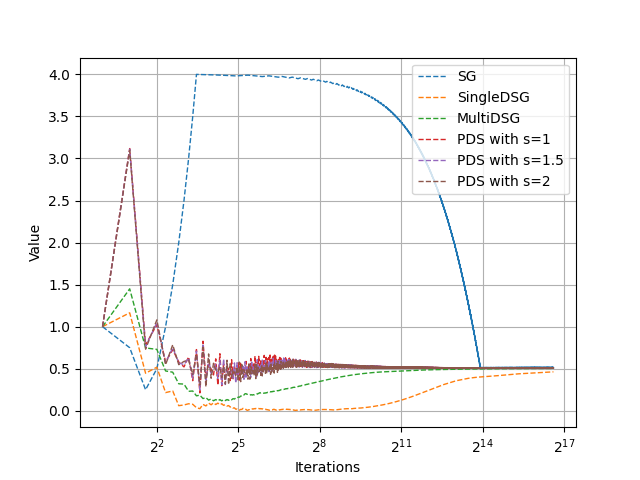}}\hfill
  {\includegraphics[width=.5\linewidth]{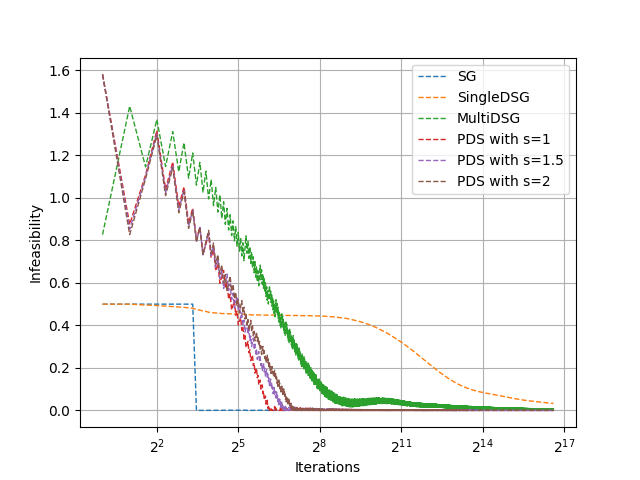}}\hfill
\caption{Illustration for MAD8 problem (Id 7).}
\label{fig:mad8}

  {\includegraphics[width=.5\linewidth]{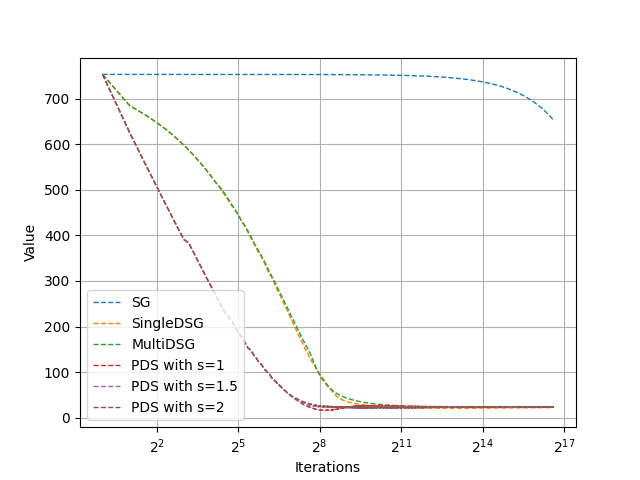}}\hfill
  {\includegraphics[width=.5\linewidth]{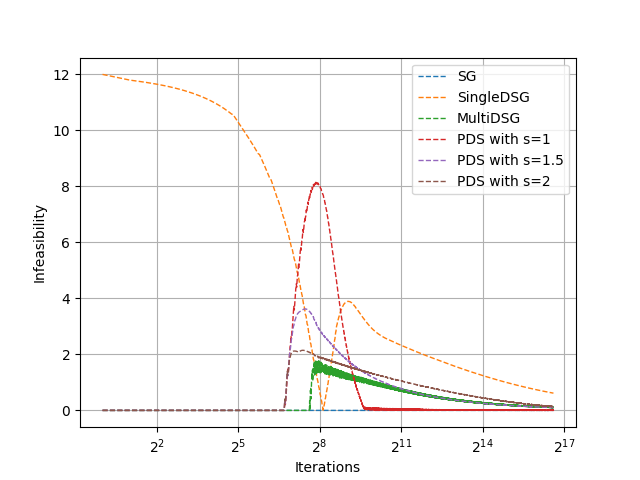}}\hfill
\caption{Illustration for Wong2 problem (Id 8).}
\label{fig:wong2}
  {\includegraphics[width=.5\linewidth]{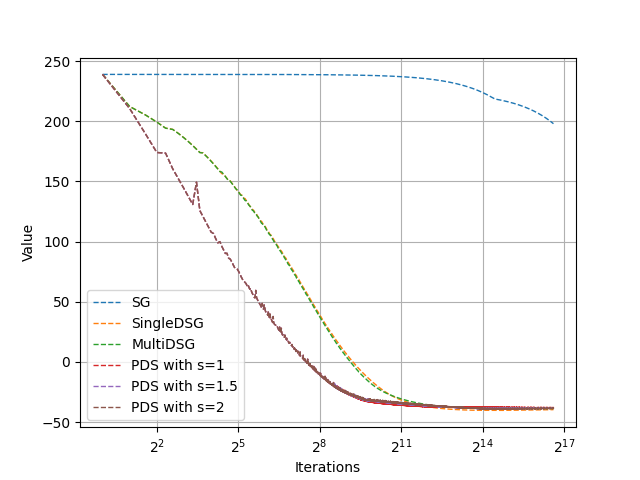}}\hfill
  {\includegraphics[width=.5\linewidth]{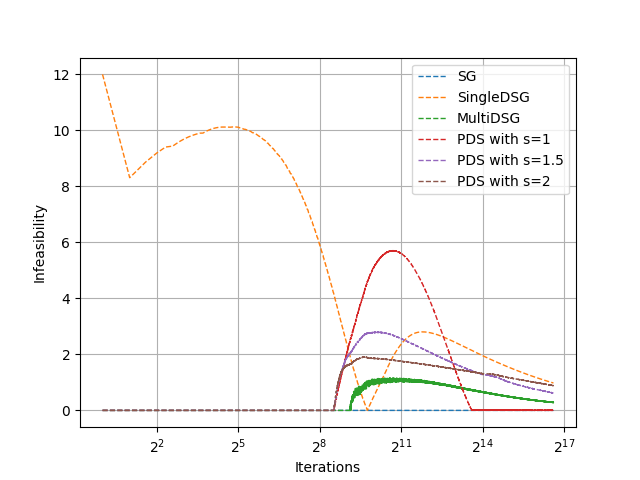}}\hfill
\caption{Illustration for Wong3 problem (Id 9).}
\label{fig:wong3}
\end{figure}
\end{center}

In Table \ref{tab:minimax}, the value returned by PDS with $s=1$ has the smallest gap w.r.t the exact value, while it takes less time compared to SG, SingleDSG, and MultiDSG. 
Moreover, the infeasibility of PDS with $s=1$ converges to zero more efficiently than SG, SingleDSG, and MultiDSG when the number of iterations increases.

Figures \ref{fig:lad.n10}, \ref{fig:lad.n100}, and \ref{fig:lad.n1000} show that the values returned by PDS with $s\in\{1,1.5,2\}$ converge at the same rate in this case, but the convergence of the infeasibility of PDS with $s=1$ to zero is the fastest.

\subsection{Least absolute deviations (LAD)}
Consider the following problem:
\begin{equation}\label{eq:lad}
        \min\limits_{x\in\R^{\overline n}} \|Dx-w\|_1\,.
\end{equation}
where $D \in \mathbb{R}^{\overline m\times \overline n}$ and $w\in \mathbb{R}^{\overline m}$ have entries taken in $[-1,1]$ with uniform distribution.

By adding slack variables $y=Dx-w$, \eqref{eq:lad} is equivalent to the CCOP:
\begin{equation}\label{eq:lad.equi}
\begin{array}{rl}
        \min\limits_{x,y} &\|y\|_1\\
        \text{s.t.}& y=Dx-w\,.
        \end{array}
\end{equation}
Let us solve \eqref{eq:lad.equi} by using SG, SingleDSG, MultiDSG and PDS.
The size of the test problems and the setting of our software are given in Table
\ref{tab:lad.prob}.
\begin{table}
\caption{\small Least absolute deviations}
    \label{tab:lad.prob}
    {\small\begin{itemize}
        \item Setting: $n=3\overline n$, $m=0$, $l=\overline m=2\overline n$, $K=10^5$, $\varepsilon=10^{-3}$, $\rho=1/s$, $\delta=0.99$.
    \end{itemize}}
\small
\begin{center}
   \begin{tabular}{|c|c|c|c|c|}
        \hline
        \multirow{2}{*}{Id}&
        \multicolumn{3}{c|}{size}\\ \cline{2-4}
 &  $\overline n$ & $n$ & $l$\\
\hline  
10 &10 & 30 & 20 \\ \hline
11 &100&300 &200\\ \hline
12 &1000&3000 & 2000\\ 
\hline
\end{tabular}    
\end{center}
\end{table}
The numerical results are  displayed in Table \ref{tab:lad}.
\begin{table}
\caption{\small Numerical results for least absolute deviations}
    \label{tab:lad}
\small
\begin{center}
   \begin{tabular}{|c|c|c|c|c|c|c|c|c|c|}
   \hline
        \multirow{2}{*}{Id}&
        \multicolumn{3}{c|}{SG} & \multicolumn{3}{c|}{SingleDSG }&      \multicolumn{3}{c|}{MultiDSG} \\ \cline{2-10}
 & val & infeas &time&
val&  infeas &time &
val & infeas & time\\
\hline  
10 & 0.0004 & 0.0010 & 23 & 0.0030 & 0.0017 & 39 &0.0017 & 0.0013& 21\\ \hline
 11 & 51.917& 0.0014 & 252 &0.1084&0.0641&507 &0.0945& 0.0013&108\\ \hline
12& 810.46 & 0.5941 &3682 & 87.849 & 8.4173 & 7578 & 11.535 & 0.0013 & 1858 \\ 
\hline
        \hline
        \multirow{2}{*}{Id}&
        \multicolumn{3}{c|}{PDS with $s=1$} & \multicolumn{3}{c|}{PDS with $s=1.5$ }&      \multicolumn{3}{c|}{PDS with $s=2$} \\ \cline{2-10}
 & val & infeas &time&
val&  infeas &time &
val & infeas & time\\
\hline  
10 & 0.0065 & 0.0033 & 14 & 0.0062 &0.0017 & 15 & 0.0074 & 0.0019 & 14\\ \hline
 11 & 0.0159& 0.0151 & 132 &0.0222&0.0021&116&0.0237&0.0020&120\\ \hline
12& 0.0434 & 43.281 &1787 & 0.0749 & 0.0044 & 1520 & 0.0739 & 0.0020 & 1443 \\ 
\hline
\end{tabular}    
\end{center}
\end{table}
The convergences of SG, SingleDSG, MultiDSG and PDS with $s\in\{1,1.5,2\}$ are illustrated in Figures \ref{fig:mad8}, \ref{fig:wong2} and \ref{fig:wong3}.
\begin{center}
    \begin{figure}
  {\includegraphics[width=.5\linewidth]{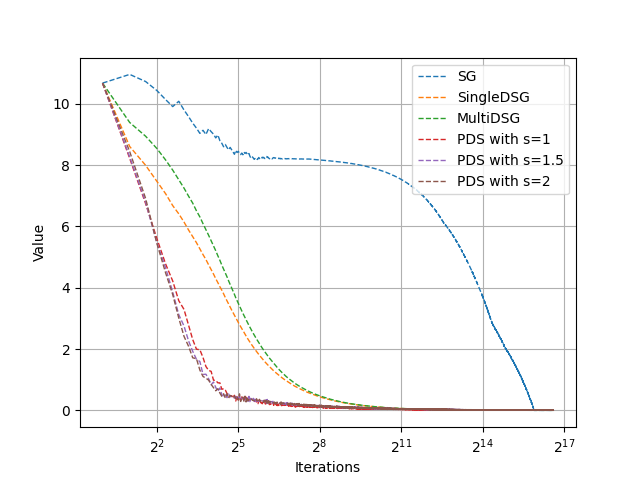}}\hfill
  {\includegraphics[width=.5\linewidth]{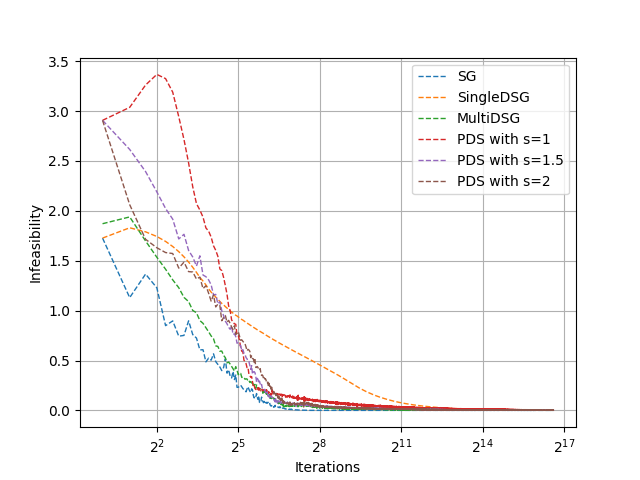}}\hfill
\caption{Illustration for LAD with $\overline n=10$ (Id 10).}
\label{fig:lad.n10}
  {\includegraphics[width=.5\linewidth]{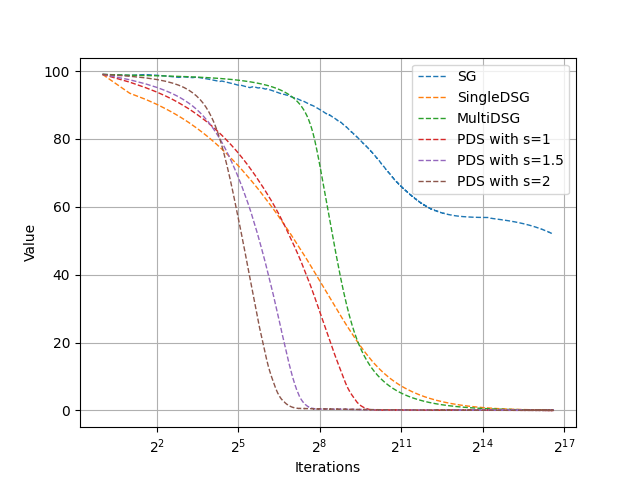}}\hfill
  {\includegraphics[width=.5\linewidth]{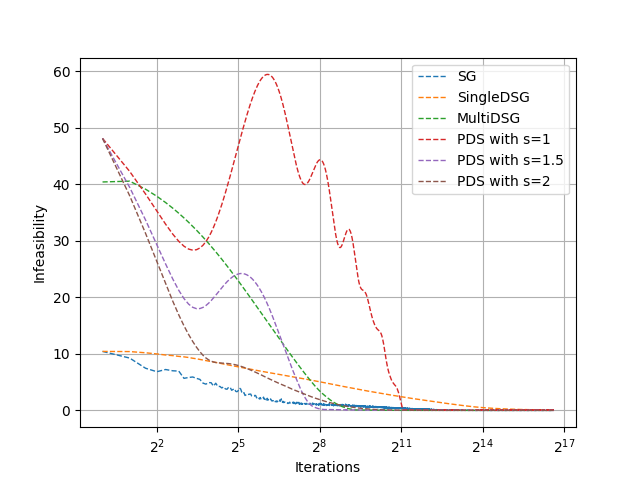}}\hfill
\caption{Illustration for LAD with $\overline n=100$ (Id 11).}
\label{fig:lad.n100}

  {\includegraphics[width=.5\linewidth]{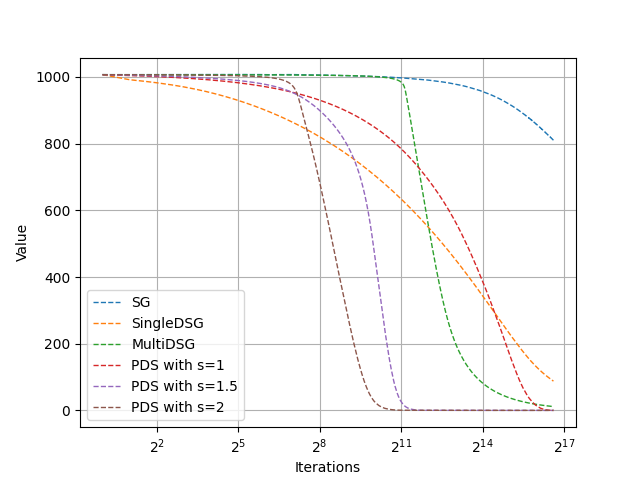}}\hfill
  {\includegraphics[width=.5\linewidth]{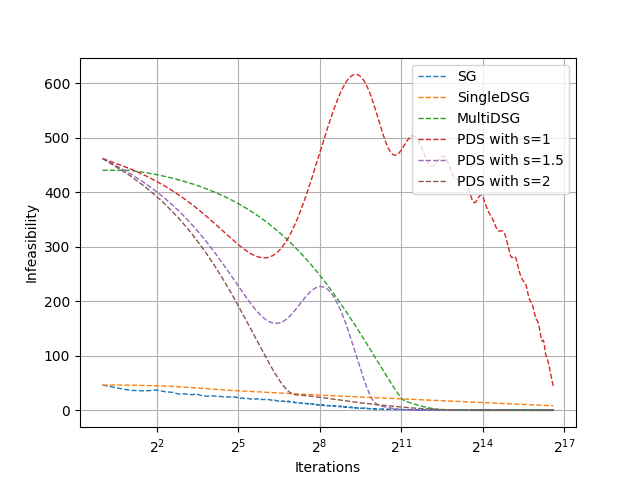}}\hfill
\caption{Illustration for LAD with $\overline n=1000$ (Id 12).}
\label{fig:lad.n1000}
\end{figure}
\end{center}
The value returned by PDS with $s=2$ has the the fastest convergence rate while requires smaller total time compared to SG, SingleDSG and MultiDSG in all cases. 
Moreover PDS with $s=2$ returns smallest infeasibility at the final iteration.

\subsection{Support vector machine (SVM)}
Consider the following problem:
\begin{equation}\label{eq:SVM}
   \min_{w,u} \left[{{\frac {1}{N}}\sum _{i=1}^{N}\max\{0,1-y_{i}(z_i ^\top w -u)\}+\frac{1}{2} \|w\|^{2}_2}\right]\,,
\end{equation}
where $z_i \in \R^{\overline n}$ and $y_i \in \{-1,1\}$ are taken as follows:
\begin{itemize}
    \item For $i=1,\dots,\lfloor N/2\rfloor$, we choose $y_i=1$ and take $z_i$ in $[0,1]^{\overline{n}}$ with uniform distribution.
    \item For $i=\lfloor N/2\rfloor+1,\dots,N$, we choose $y_i=-1$ and take $z_i$  in $[-1,0]^{\overline{n}}$ with uniform distribution.
\end{itemize}

Letting $\tau_i=z _{i}^\top w-u$, problem \eqref{eq:SVM} is equivalent to the problem:
\begin{equation}\label{eq:SVM.equi}
\begin{array}{rl}
   \min\limits_{w,u,\tau}& {{\frac {1}{N}}\sum _{i=1}^{N}\max\{0,1-y_{i}\tau_i\}+\frac{1}{2} \|w\|^{2}_2}\\
   \text{s.t.}& \tau = Zw-ue\,,
   \end{array}
\end{equation}
where $Z=\begin{bmatrix}
z_1^\top\\
\dots\\
z_N^\top\end{bmatrix}$ and $e=\begin{bmatrix}
1\\
\dots\\
1\end{bmatrix}$.

Let us solve \eqref{eq:SVM.equi}  by using SG, SingleDSG, MultiDSG and PDS.
The size of test problems and the setting of our software are given in Table \ref{tab:SVM.prob}.
\begin{table}
 \caption{\small Support vector machine}
    \label{tab:SVM.prob}
    {\small\begin{itemize}
        \item Setting: $N=200\times {\overline n}$, $m=0$, $\varepsilon=10^{-3}$, $\delta=0.5$, $\rho=1/s$, $\delta=0.99$, $K=10^4$.
    \end{itemize}}
\footnotesize
\begin{center}
   \begin{tabular}{|c|c|c|c|}
        \hline
        \multirow{2}{*}{Id}&
        \multicolumn{3}{c|}{Size}
         \\ \cline{2-4}
&$\overline n$ & $n$ & $l$\\
\hline  
13&2 & 403 & 400 \\ \hline
14&3&604 & 600\\ 
\hline
15&5&1006 & 1000\\ 
\hline
\end{tabular}    
\end{center}
\end{table}
The numerical results are displayed in Table \ref{tab:SVM}.
\begin{table}
 \caption{\small Numerical results for support vector machine}
    \label{tab:SVM}
\footnotesize
\begin{center}
   \begin{tabular}{|c|c|c|c|c|c|c|c|c|c|}
        \hline
        \multirow{2}{*}{Id}&
        \multicolumn{3}{c|}{SG} & \multicolumn{3}{c|}{SingleDSG}&      \multicolumn{3}{c|}{MultiDSG} \\ \cline{2-10}
& val & infeas &time&
val&  infeas &time &
val & infeas & time\\
\hline  
13 & 0.9764 & 0.0014 & 254 & 0.9823 & 0.2576 & 316 &0.9956 &0.0610&268\\ \hline
14& 0.9821 & 0.0012 & 516 & 0.9793 & 0.3068 & 610 & 0.9930 & 1.3704 & 515 \\ 
\hline
15& 0.9978 & 0.0014 & 778 & 0.9892 & 0.3616 & 991 & 0.9979 & 6.0966 & 846 \\ 
\hline
\hline
        \multirow{2}{*}{Id}&
        \multicolumn{3}{c|}{PDS with $s=1$} & \multicolumn{3}{c|}{PDS with $s=1.5$ }&      \multicolumn{3}{c|}{PDS with $s=2$} \\ \cline{2-10}
& val & infeas &time&
val&  infeas &time &
val & infeas & time\\
\hline  
13 & 0.9574 & 0.0998 & 140 & 0.9492 & 0.0957 &139 &0.9019& 0.0955& 144\\ \hline
14& 0.9678 & 3.2489 &229 & 0.9612 & 0.1171 & 255 & 0.9436 & 0.1170 & 246 \\ 
\hline
15& 0.9945 & 12.655 &430& 0.9823 &  0.1954 & 452 & 0.9831 & 0.1736 & 422 \\ 
\hline
\end{tabular}    
\end{center}
\end{table}
Figures \ref{fig:svm.n2}, \ref{fig:svm.n3} and \ref{fig:svm.n5} show the progress of SG, SingleDSG, MultiDSG and PDS with $s\in\{1,1.5,2\}$.

\begin{center}
    \begin{figure}
  {\includegraphics[width=.5\linewidth]{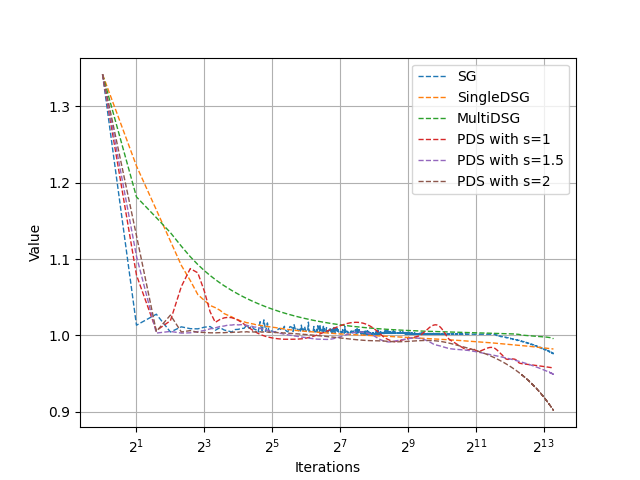}}\hfill
  {\includegraphics[width=.5\linewidth]{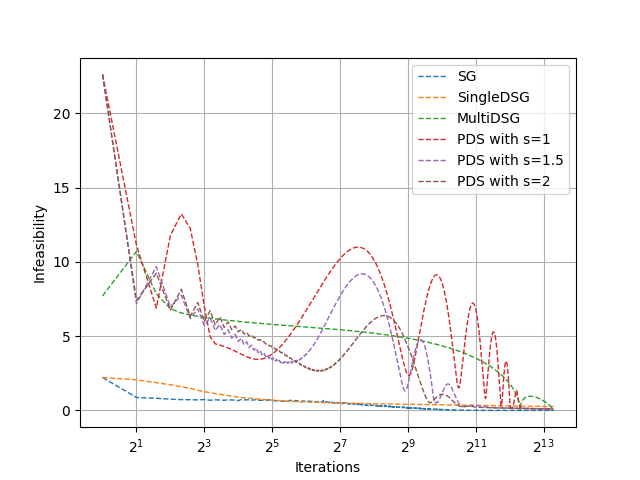}}\hfill
\caption{Illustration for SVM with $\overline n=2$ (Id 13).}
\label{fig:svm.n2}

  {\includegraphics[width=.5\linewidth]{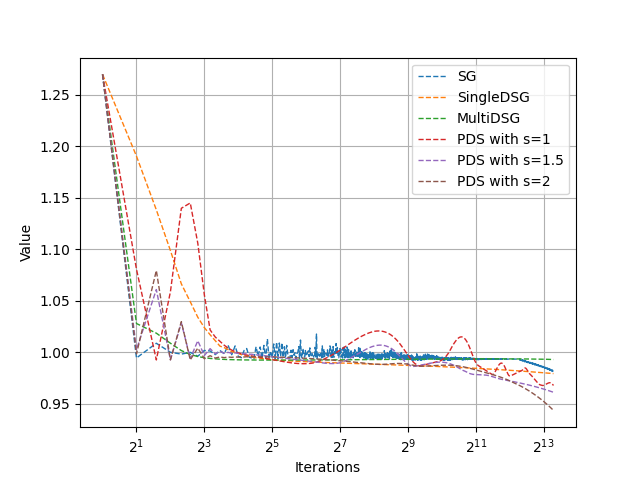}}\hfill
  {\includegraphics[width=.5\linewidth]{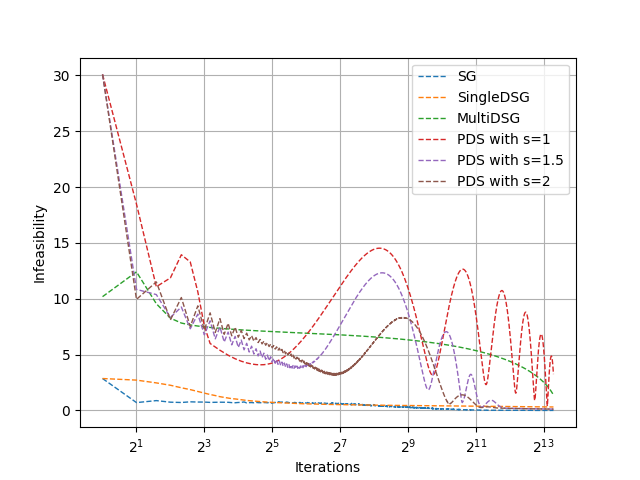}}\hfill
\caption{Illustration for SVM with $\overline n=3$ (Id 14).}
\label{fig:svm.n3}

  {\includegraphics[width=.5\linewidth]{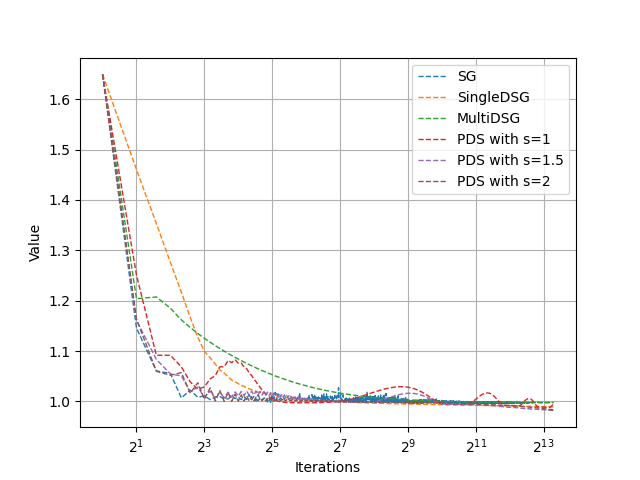}}\hfill
  {\includegraphics[width=.5\linewidth]{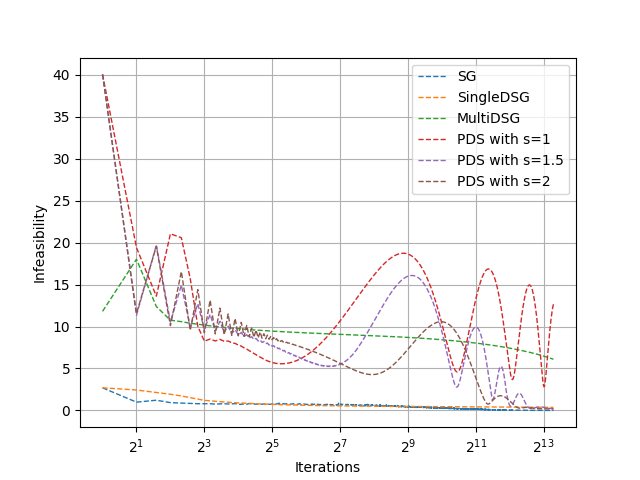}}\hfill
\caption{Illustration for SVM with $\overline n=5$ (Id 15).}
\label{fig:svm.n5}
\end{figure}
\end{center}
\section{Conclusion}
We have tested the performance of different subgradient methods in solving COP with functional constraints.
We  emphasize that PDS with $s\in[1,2]$ and MultiDSG are the best choices for users.
They provide better approximations in less computational time.
For COP \eqref{eq:CCOP} with the number of inequality constraints larger than the number of equality ones ($m> l$), users should probably choose PDS with $s$ close to $1$ (see MAD8, Wong2, and Wong3).
When the number of equality constraints of COP \eqref{eq:CCOP} is much greater than the number of inequality ones ($l\gg m$), PDS with $s$ close to $2$ would be a good choice (see Case 1 and LAD).
For COP \eqref{eq:CCOP} with a large number of equality constraints and a large number of inequality ones ($m\gg 1$ and $l\gg 1$), the best method would probably be MultiDSG (see Case 2).
SG and SingleDSG are much more inefficient than the others because they use the alternative single constraint $\overline{f}(x)\le 0$ with $\overline{f}$ defined as in \eqref{eq:alter.const}, and hence lose the information of the dual problem.

As a topic of further applications, we would like to use MultiDSG and PDS for solving large-scale semidefinite programs (SDPs) of form:
\begin{equation}\label{eq:SDP}
p^\star \,:= \,\inf\limits_{y \in \R^{n} }\left\{  f^\top y\ \left|\begin{array}{rl}
&C_i+\mathcal{B}_i y\preceq 0 \,,\,i=1,\dots, m\,,\\
& Ay=b
\end{array}
\right. \right\}\,,
\end{equation}
where $f\in \R^n$, $C_i\in \S^{s_i}$ $\mathcal{B}_i:\R^n\to \S^{s_i}$ is a linear operator defined by $\mathcal{B}_i y=\sum_{j=1}^n y_jB_{i}^{(j)}$ with $B_{i}^{(j)}\in \S^{s_i}$, $A\in\R^{l\times n}$, and $b\in\R^l$.
Here we denote by $\S^s$ the set of real symmetric matrices of size $s$.
Obviously, SDP \eqref{eq:SDP} is equivalent to COP with functional constraints:
\begin{equation}\label{eq:SDP.NS}
p^\star \,:= \,\inf\limits_{y \in \R^n }\left\{  f^\top y\ \left|\begin{array}{rl}
&\lambda_{\max}(C_i+\mathcal{B}_iy)\le 0 \,,\,i=1,\dots, m\,,\\
& Ay=b
\end{array}
\right. \right\}\,,
\end{equation}
where $\lambda_{\max}(A)$ stands for the largest eigenvalue of a given symmetric matrix $A$.

Finally, another application is to solve large-scale nonsmooth optimization problems arising from machine learning using MultiDSG and PDS.

\paragraph{\textbf{Acknowledgements}.} 
The first author was supported by the  funding from Torus-Actions.
The second author was supported by the MESRI funding from EDMITT.
%
\appendix

\section{Appendix}
\label{sec:Appendix}
\subsection{Dual subgradient method}
In each iteration $z^{(k+1)}=z^{(0)}-\frac{s^{(k+1)}}{\beta_{k}}$ is the maximizer of 
\begin{equation}\label{ufunc}
U^s_{\beta}(z):=- s^\top(z-z^{(0)})-\frac{\beta}{2}\|z-z^{(0)}\|^2_2
\end{equation}

for $s=s^{(k+1)}$ and $\beta=\beta_k$, with 
\begin{equation}\label{eq:optval}
U_{\beta_k}^{s^{(k+1)}}(z^{(k+1)})=\frac{\|s^{(k+1)}\|^2_2}{2\beta_k}.
\end{equation}

In addition, $U^s_{\beta}(z)$ is strongly concave in $z$ with parameter $\beta$,
\begin{equation}\label{strcave}
U_{\beta}^s(z)\le U_{\beta}^s(z')+ \nabla U_{\beta}^s(z')^\top(z-z') 
-\frac{\beta}{2}\|z-z'\|^2_2.
\end{equation}

We define $\overline{G}^{(k)}:=\frac{G(z^{(k)})}{\|G(z^{(k)})\|_2}$. 
Note that $\|\overline{G}^{(k)}\|_2=1$.

\begin{lemma}
	\label{bounded}
In method \eqref{alg:dsg}, it holds that  $\|z^{(k)}-z^\star\|_2\le\|z^{(0)}-z^\star\|_2+1$.
\end{lemma}

\begin{proof} 	

From \eqref{eq:optval}, 
\begin{equation}
\begin{array}{rl}
U_{\beta_k}^{s^{(k+1)}}(z^{(k+1)})&=\frac{\beta_{k-1}}{\beta_k}\frac{\|s^{(k+1)}\|^2_2}{2\beta_{k-1}}=\frac{\beta_{k-1}}{\beta_k}\frac{\|s^{(k)}+\overline{G}^{(k)}\|^2_2}{2\beta_{k-1}}\\[5pt]
&=\frac{\beta_{k-1}}{\beta_k}\left(\frac{\|s^{(k)}\|^2_2}{2\beta_{k-1}}+\frac{1}{\beta_{k-1}}
s^{(k)\top}\overline{G}^{(k)}+\frac{\|\overline{G}^{(k)}\|^2_2}{2\beta_{k-1}}\right)\\[10pt]
&=\frac{\beta_{k-1}}{\beta_k}\left(U_{\beta_{k-1}}^{s^{(k)}}(z^{(k)})+\frac{1}{\beta_{k-1}}
s^{(k)\top}\overline{G}^{(k)}+\frac{1}{2\beta_{k-1}}\right)\\[10pt]
&=\frac{\beta_{k-1}}{\beta_k}\left(U_{\beta_{k-1}}^{s^{(k)}}(z^{(k)})+( 
z^{(0)}-z^{(k)})^\top\overline{G}^{(k)}+\frac{1}{2\beta_{k-1}}\right).
\end{array}
\end{equation}
Rearranging,
\begin{equation}
\begin{array}{rl}
(z^{(k)}-z^{(0)})^\top\overline{G}^{(k)}
&=U_{\beta_{k-1}}^{s^{(k)}}(z^{(k)})-\frac{\beta_k}{\beta_{k-1}}U_{\beta_k}^{s^{(k+1)}}(z^{(k+1)})+\frac{1}{2\beta_{k-1}}\\[10pt]
&\le U_{\beta_{k-1}}^{s^{(k)}}(z^{(k)})-U_{\beta_k}^{s^{(k+1)}}(z^{(k+1)})+\frac{1}{2\beta_{k-1}},
\end{array}
\end{equation}
since $\beta_k$ is increasing. Telescoping these inequalities for 
$k=1,\dots,K$, and using the fact that $\|s_1\|_2=1$,
\begin{equation}
\begin{array}{rl}
\sum_{k=1}^K
(z^{(k)}-z^{(0)})^\top\overline{G}^{(k)}&\le 
U_{\beta_{0}}^{s_1}(z^{(1)})-U_{\beta_K}^{s^{(K+1)}}(z^{(K+1)})+\sum_{k=1}^K\frac{1}{2\beta_{k-1}}\label{kgo}\\[10pt]
&=\frac{1}{2\beta_0}-U_{\beta_K}^{s^{(K+1)}}(z^{(K+1)})+\sum_{k=0}^{K-1}\frac{1}{2\beta_k}\\[10pt]
&=-U_{\beta_K}^{s^{(K+1)}}(z^{(K+1)})+\frac{1}{2}\left(\sum_{k=0}^{K-1}\frac{1}{\beta_k}+\beta_0\right).
\end{array}
\end{equation}

Expanding the recursion $\beta_k=\frac{1}{\beta_k-1}+\beta_{k-1}$,
\begin{equation}
\label{constart}
\sum_{k=1}^K 
(z^{(k)}-z^{(0)})^\top\overline{G}^{(k)}\le
-U_{\beta_K}^{s^{(K+1)}}(z^{(K+1)})+\frac{\beta_K}{2}.
\end{equation}  

Given the convexity of $L(x,\lambda,\nu)$ in $x$ and linearity in $(\lambda,\nu)$,
\begin{equation}\label{4-1}
    L(x^\star,\lambda^{(k)},\nu^{(k)})\ge L(x^{(k)},\lambda^{(k)},\nu^{(k)})+  G_x(x^{(k)},\lambda^{(k)},\nu^{(k)})^\top(x^\star-x^{(k)})\,,
\end{equation}
and
\begin{equation}\label{4-2}
    \begin{array}{rl}
         L(x^{(k)},\lambda^\star,\nu^\star)=& L(x^{(k)},\lambda^{(k)},\nu^{(k)})+ 
G_\lambda(x^{(k)},\lambda^{(k)},\nu^{(k)})^\top(\lambda^\star-\lambda^{(k)})\rangle\\
&+ 
G_\nu(x^{(k)},\lambda^{(k)},\nu^{(k)})^\top(\nu^\star-\nu^{(k)}).
    \end{array}
\end{equation}

Subtracting \eqref{4-1} from \eqref{4-2} and using \eqref{eq:dual.CCOP},
\begin{equation}
\label{conineq}
0\le  (G_x(z^{(k)}),-G_\lambda(z^{(k)}),-G_\nu(z^{(k)}))^\top (z^{(k)}-z^\star).
\end{equation}  

It follows that 
\begin{equation}\label{strcon}
\begin{array}{rl}
0\le&\sum_{k=1}^K(z^{(k)}-z^\star)^\top\overline{G}^{(k)}\\[5pt]
=&\sum_{k=1}^K(z^{(0)}-z^\star)^\top\overline{G}^{(k)}+\sum_{k=1}^K(
z^{(k)}-z^{(0)})^\top\overline{G}^{(k)}\\[5pt]
\le&\sum_{k=1}^K 
(z^{(0)}-z^\star)^\top\overline{G}^{(k)}-U_{\beta_K}^{s^{(K+1)}}(z^{(K+1)})+\frac{\beta_K}{2}\\[5pt]
=& 
(z^{(0)}-z^\star)^\top s^{(k+1)}-U_{\beta_K}^{s^{(K+1)}}(z^{(K+1)})+\frac{\beta_K}{2}\,,
\end{array}
\end{equation}
where the second inequality uses \eqref{constart}, and the second equality follows since 
$s^{(k+1)}=s^{(k)}+\overline{G}^{(k)}$. Considering inequality \eqref{strcave} with $s=s^{(K+1)}$, 
$\beta=\beta_{K}$, $z=z^\star$, and $z'=z^{(K+1)}$,
\begin{equation}
\begin{array}{rl}
U_{\beta_K}^{s^{(K+1)}}(z^\star)\le& U_{\beta_K}^{s^{(K+1)}}(z^{(K+1)})+ \nabla 
U_{\beta_K}^{s^{(K+1)}}(z^{(K+1)})^\top(z^\star-z^{(K+1)})\\[5pt]
&-\frac{\beta_K}{2}\|z^\star-z^{(K+1)}\|^2_2\\[5pt]
=& U_{\beta_K}^{s^{(K+1)}}(z^{(K+1)})-\frac{\beta_K}{2}\|z^\star-z^{(K+1)}\|^2_2,
\end{array}
\end{equation}
given that $z^{(K+1)}$ is the maximum of $U_{\beta_K}^{s^{(K+1)}}(z)$.
Applying this inequality in \eqref{strcon},
\begin{equation}
    \begin{array}{rl}
         0\le& 
(z^{(0)}-z^\star)^\top s^{(k+1)}-U_{\beta_K}^{s^{(K+1)}}(z^\star)-\frac{\beta_K}{2}\|z^\star-z^{(K+1)}\|^2_2+\frac{\beta_K}{2}\\[5pt]
=&
(z^{(0)}-z^\star)^\top s^{(k+1)}+ 
s^{(k+1)\top} (z^\star-z^{(0)})+\frac{\beta_K}{2}\|z^\star-z^{(0)}\|^2_2\\[5pt]
&-\frac{\beta_K}{2}\|z^\star-z^{(K+1)}\|^2_2+\frac{\beta_K}{2}\\[5pt]
=&\frac{\beta_K}{2}\|z^\star-z^{(0)}\|^2_2-\frac{\beta_K}{2}\|z^\star-z^{(K+1)}\|^2_2+\frac{\beta_K}{2},
    \end{array}
\end{equation}
where the first equality uses the definition of $U_{\beta_K}^{s^{(K+1)}}(z^\star)$ \eqref{ufunc}. 
Rearranging,
\begin{equation}
\|z^\star-z^{(K+1)}\|^2_2\le\|z^\star-z^{(0)}\|^2_2+1.\label{normsq}
\end{equation}
As $K\ge 1$ from \eqref{kgo}, this implies that \eqref{normsq} holds for $K\ge 2$. Considering now 
when $k=1$,
\begin{equation}
\begin{array}{rl}
\|z^{(1)}-z^\star\|^2_2=&\|z^{(0)}-z^\star-\overline{G}^{(0)}\|^2_2=\|z^{(0)}-z^\star\|^2_2-2 (z^{(0)}-z^\star)^\top
\overline{G}^{(0)}+1\\
\le&\|z^{(0)}-z^\star\|^2_2+1\,,
\end{array}
\end{equation}
where the last line uses \eqref{conineq}. Now for all $k$, 
\begin{equation}
(\|z^\star-z^{(0)}\|_2+1)^2=\|z^\star-z^{(0)}\|^2_2+2\|z^\star-z^{(0)}\|+1\ge\|z^\star-z^{(k)}\|^2_2+2\|z^\star-z^{(0)}\|\,,\nonumber
\end{equation}
so that $\|z^\star-z^{(0)}\|_2+1\ge\|z^\star-z^{(k)}\|_2$.
\end{proof}

In order to prove the convergence result of Algorithm \ref{alg:dsg}, we require bounding the norm of the subgradients $G(z^{(k)})$. 

\begin{lemma}
\label{gradbound}
There exists a constant $C>0$ such that 
$\|G(z^{(k)})\|_2\le C$, for all $k\in\N$.
\end{lemma}

\begin{proof} 	
Recall that $g_0(x)\in \partial f_0(x)$ and $g_i(x)\in \partial F_i(x)$, 
\begin{equation}
\begin{array}{rl}
\|G(z^{(k)})\|_2=&\|(G_x(z^{(k)}),-G_{\lambda}(z^{(k)}),-G_{\nu}(z^{(k)}))\|_2\\[5pt]
=&\left\|\left(g(x^{(k)})+\sum_{i=1}^m\lambda_i^{(k)}g_i(x^{(k)}) +A^\top\nu^{(k)},-F(x^{(k)}),-Ax^{(k)}+b\right)\right\|_2\\[5pt]
\le&\|g(x^{(k)})\|_2+\sum_{i=1}^m\lambda_i^{(k)}\|g_i(x^{(k)})\|_2\\[5pt]
&+\|A^\top\|\|\nu^{(k)}\|_2+F(x^{(k)})+\|Ax^{(k)}-b\|_2.
\end{array}
\end{equation}
Here we note $\|A^\top\|:=\max_{u\in\R^l}\{\|A^\top u\|_2/\|u\|_2\}$.
The iterates of method \eqref{alg:dsg} are bounded in a convex compact region, 
$z^{(k)}\in 
D:=\{z : \|z-z^\star\|_2\le \|z^{(0)}-z^\star\|_2+1\}$. This implies that $x^{(k)}\in 
D_x:=\{x : \|x-x^\star\|_2\le \|z^{(0)}-z^\star\|_2+1\}$, $\lambda^{(k)}\in
D_{\lambda}:=\{\lambda : \|\lambda-\lambda^\star\|_2\le \|z^{(0)}-z^\star\|_2+1\}$ and $\nu^{(k)}\in 
D_{\nu}:=\{\nu : \|\nu-\nu^\star\|_2\le \|z^{(0)}-z^\star\|_2+1\}$. 
The desired result follows from Assumption \ref{ass:bound.subgrad}.
\end{proof}

\begin{lemma}
	\label{convergence}
	There exists a constant $C>0$ such that for any $K\in\N$ iterations in method  \eqref{alg:dsg},
	\begin{equation}
	f_0(\overline{x}^{(K+1)})-p^\star\le 
	\frac{C(\|z^{(0)}-z^\star\|^2_2+1)}{2(K+1)}\left(\frac{1}{1+\sqrt{3}}+\sqrt{2K+1}\right)\nonumber
	\end{equation}
	and
	\begin{equation}
\|F(\overline{x}^{(K+1)})\|_2+\|A\overline{x}^{(K+1)}-b\|_2\le 
	\frac{C(4(\|z^{(0)}-z^\star\|_2+1)^2+1)}{2(K+1)}\left(\frac{1}{1+\sqrt{3}}+\sqrt{2K+1}\right).\nonumber
	\end{equation} 
\end{lemma}
\begin{proof} 	
Using equation \eqref{constart}, and recalling that $z^{(K+1)}$ maximizes $U_{\beta_K}^{s^{(K+1)}}(z)$ defined by
\eqref{ufunc}, 
\begin{equation}\label{prfin}
\begin{array}{rl}
\frac{\beta_K}{2}\ge&\sum_{k=1}^K
(z^{(k)}-z^{(0)})^\top\overline{G}^{(k)}+U_{\beta_K}^{s^{(K+1)}}(z^{(K+1)})\\[5pt]
=&\sum_{k=1}^K(
z^{(k)}-z^{(0)})^\top\overline{G}^{(k)}\\[5pt]
&+\max\limits_{z\in\R^{n+m+l}}\left\{-\left(
\sum_{k=0}^K\overline{G}^{(k)}\right)^\top(z-z^{(0)})-\frac{\beta_K}{2}\|z-z^{(0)}\|^2_2\right\}\\[5pt]
=&\max\limits_{z\in\R^{n+m+l}}\left\{
\sum_{k=0}^K-
\overline{G}^{(k)\top}(z-z^{(k)})-\frac{\beta_K}{2}\|z-z^{(0)}\|^2_2\right\}.
\end{array} 
\end{equation}
Like $\overline{x}^{(K+1)}$, let 
$\overline{z}^{(K+1)}:=\delta_{K+1}^{-1}\sum_{k=0}^K\frac{z^{(k)}}{\|G(z^{(k)})\|_2}$, $\overline{\lambda}^{(K+1)}:=\delta_{K+1}^{-1}\sum_{k=0}^K\frac{\lambda^{(k)}}{\|G(z^{(k)})\|_2}$ and $\overline{\nu}^{(K+1)}:=\delta_{K+1}^{-1}\sum_{k=0}^K\frac{\nu^{(k)}}{\|G(z^{(k)})\|_2}$.
Multiplying both sides of \eqref{prfin} by $\delta_{K+1}^{-1}$,
\begin{equation}\label{maxx}
\begin{array}{rl}
&\delta_{K+1}^{-1}\frac{\beta_K}{2}\\[10pt]
\ge&\delta_{K+1}^{-1}\max\limits_{z\in\R^{n+m+l}}\left\{
\sum_{k=0}^K-
\overline{G}^{(k)\top}(z-z^{(k)})-\frac{\beta_K}{2}\|z-z^{(0)}\|^2_2\right\}\\[10pt]
=&\delta_{K+1}^{-1}\max\limits_{z\in\R^{n+m+l}}\left\{
\sum_{k=0}^K- 
\frac{G_x(z^{(k)})^\top}{\|G(z^{(k)})\|_2}(x-x^{(k)})+ 
\frac{G_{\lambda}(z^{(k)})^\top}{\|G(z^{(k)})\|_2}(\lambda-\lambda^{(k)})\right.\\[10pt]
&+ 
\left.\frac{G_{\nu}(z^{(k)})^\top}{\|G(z^{(k)})\|_2}(\nu-\nu^{(k)})-\frac{\beta_K}{2}\|z-z^{(0)}\|^2_2\right\}\\[5pt]
\ge&\delta_{K+1}^{-1}\max\limits_{z\in\R^{n+m+l}}\left\{ 
\sum_{k=0}^K 
\frac{L(x^{(k)},\lambda^{(k)},\nu^{(k)})-L(x,\lambda^{(k)},\nu^{(k)})}{\|G(z^{(k)})\|_2}\right.\\[10pt]
&\left.+\frac{L(x^{(k)},\lambda,\nu)-L(x^{(k)},\lambda^{(k)},\nu^{(k)})}{\|G(z^{(k)})\|_2}-\frac{\beta_K}{2}\|z-z^{(0)}\|^2_2\right\}\\[10pt]
=&\delta_{K+1}^{-1}\max\limits_{z\in\R^{n+m+l}}\left\{
\sum_{k=0}^K 
\frac{L(x^{(k)},\lambda,\nu)-L(x,\lambda^{(k)},\nu^{(k)})}{\|G(z^{(k)})}_2\|-\frac{\beta_K}{2}\|z-z^{(0)}\|^2_2\right\}\\[10pt]
=&\max\limits_{z\in\R^{n+m+l}}\left\{ 
\delta_{K+1}^{-1}\sum_{k=0}^K 
\frac{L(x^{(k)},\lambda,\nu)-L(x,\lambda^{(k)},\nu^{(k)})}{\|G(z^{(k)})\|_2}-\delta_{K+1}^{-1}\frac{\beta_K}{2}\|z-z^{(0)}\|^2_2\right\}\\[10pt]
\ge&\max\limits_{z\in\R^{n+m+l}}\left\{ 
L(\overline{x}^{(K+1)},\lambda,\nu)-L(x,\overline{\lambda}^{(K+1)},\overline{\nu}^{(K+1)})-\delta_{K+1}^{-1}\frac{\beta_K}{2}\|z-z^{(0)}\|^2_2\right\}\\[5pt]
=&\max\limits_{z\in\R^{n+m+l}}\left\{ f_0(\overline{x}^{(K+1)})+F(\overline{x}^{(K+1)})^\top \lambda+(A\overline{x}^{(K+1)}-b)^\top \nu
\right.\\[10pt]
&\left.-f_0(x)-F(x)^\top \overline{\lambda}^{(K+1)}-(Ax-b)^\top\overline{\nu}^{(K+1)}-\delta_{K+1}^{-1}\frac{\beta_K}{2}\|z-z^{(0)}\|^2_2\right\},
\end{array} 
\end{equation}
where the third inequality uses Jensen's inequality. 
Given the maximum function, the inequality 
holds for any choice of $z$. 
We consider two cases, the first being $x=\overline{x}^{(K+1)}$, $\lambda=\overline{\lambda}^{(K+1)}+\frac{F(\overline{x}^{(K+1)})}{2\|F(\overline{x}^{(K+1)})\|_2}$ and $\nu=\overline{\nu}^{(K+1)}+\frac{A\overline{x}^{(K+1)}-b}{2\|A\overline{x}^{(K+1)}-b\|_2}$. 
From \eqref{maxx},
\begin{equation}\label{c11}
\begin{array}{rl}
&\delta_{K+1}^{-1}\frac{\beta_K}{2}\\[5pt]
\ge&f_0(\overline{x}^{(K+1)})+\left(\overline{\lambda}^{(K+1)}+\frac{F(\overline{x}^{(K+1)})}{2\|F(\overline{x}^{(K+1)})\|_2}\right)^\top F(\overline{x}^{(K+1)})\\[5pt]
&+\left(\overline{\nu}^{(K+1)}+\frac{A\overline{x}^{(K+1)}-b}{2\|A\overline{x}^{(K+1)}-b\|_2}\right)^\top (A\overline{x}^{(K+1)}-b)\\[5pt]
&-f_0(\overline{x}^{(K+1)})-F(\overline{x}^{(K+1)})^\top\overline{\lambda}^{(K+1)}-\overline{\nu}^{(K+1)\top} (A\overline{x}^{(K+1)}-b)\\[5pt]
&-\delta_{K+1}^{-1}\frac{\beta_K}{2}\left\|\left(\overline{x}^{(K+1)},\overline{\lambda}^{(K+1)}+\frac{F(\overline{x}^{(K+1)})}{2\|F(\overline{x}^{(K+1)})\|_2},\overline{\nu}^{(K+1)}+\frac{A\overline{x}^{(K+1)}-b}{2\|A\overline{x}^{(K+1)}-b\|_2}\right)-z^{(0)}\right\|^2_2\\[5pt]
=&\frac{1}{2}(\|F(\overline{x}^{(K+1)})\|_2+\|A\overline{x}^{(K+1)}-b\|_2)\\[5pt]
&-\delta_{K+1}^{-1}\frac{\beta_K}{2}\left\|\left(\overline{x}^{(K+1)},\overline{\lambda}^{(K+1)}
+\frac{F(\overline{x}^{(K+1)})}{2\|F(\overline{x}^{(K+1)})\|_2},\overline{\nu}^{(K+1)}+\frac{A\overline{x}^{(K+1)}-b}{2\|A\overline{x}^{(K+1)}-b\|_2}\right)-z^{(0)}\right\|^2_2.
\end{array}
\end{equation} 

Further,
\begin{equation}\label{c12}
\begin{array}{rl}
&\left\|\left(\overline{x}^{(K+1)},\overline{\lambda}^{(K+1)}+\frac{F(\overline{x}^{(K+1)})}{2\|F(\overline{x}^{(K+1)})\|_2},\overline{\nu}^{(K+1)}+\frac{A\overline{x}^{(K+1)}-b}{2\|A\overline{x}^{(K+1)}-b\|_2}\right)-z^{(0)}\right\|^2_2\\[5pt]
\le&\|\overline{z}^{(K+1)}-z^{(0)}\|_2+\frac{\|F(\overline{x}^{(K+1)})\|_2}{2\|F(\overline{x}^{(K+1)})\|_2} + \frac{\|A\overline{x}^{(K+1)}-b\|_2}{2\|A\overline{x}^{(K+1)}-b\|_2}\\[5pt]
=&\|\overline{z}^{(K+1)}-z^\star+z^\star-z^{(0)}\|_2+1\\[5pt]
\le&\|\overline{z}^{(K+1)}-z^\star\|_2+\|z^\star-z^{(0)}\|_2+1\\[5pt]
\le&\delta_{K+1}^{-1}\sum_{k=0}^K\frac{\|z^{(k)}-z^\star\|_2}{\|G(z^{(k)})\|_2}+\|z^\star-z^{(0)}\|_2+1\\[5pt]
\le&\delta_{K+1}^{-1}\sum_{k=0}^K\frac{\|z^{(0)}-z^\star\|_2+1}{\|G(z^{(k)})\|_2}+\|z^\star-z^{(0)}\|_2+1\\[5pt]
=&2(\|z^{(0)}-z^\star\|_2+1),
\end{array} 
\end{equation}
where the third inequality uses Jensen's inequality and the fourth inequality uses Lemma 
\ref{bounded}. Combining \eqref{c11} 
and \eqref{c12}, 
\begin{alignat}{6}
\|F(\overline{x}^{(K+1)})\|_2+\|A\overline{x}^{(K+1)}-b\|_2\le 
\delta_{K+1}^{-1}\frac{\beta_K}{2}(4(\|z^{(0)}-z^\star\|_2+1)^2+1).\nonumber
\end{alignat} 

The second case will use $w=z^\star$. Starting from \eqref{maxx},
\begin{equation}
\begin{array}{rl}
\delta_{K+1}^{-1}\frac{\beta_K}{2}\ge&f_0(\overline{x}^{(K+1)})+F(\overline{x}^{(K+1)})^\top \lambda^\star+(A\overline{x}^{(K+1)}-b)^\top \nu^\star\\[5pt]
&-f_0(x^\star)-F(x^\star)^\top \overline{\lambda}^{(K+1)}-(Ax^\star-b)^\top \overline{\nu}^{(K+1)}-\delta_{K+1}^{-1}\frac{\beta_K}{2}\|z^\star-z^{(0)}\|^2_2\\[5pt]
=&f_0(\overline{x}^{(K+1)})-f_0(x^\star)-\delta_{K+1}^{-1}\frac{\beta_K}{2}\|z^\star-z^{(0)}\|^2_2,
\end{array}
\end{equation} 
since $F(x^\star)=0$ and $Ax^\star=b$. Rearranging, 
\begin{equation}
f_0(\overline{x}^{(K+1)})-f_0(x^\star)\le \delta_{K+1}^{-1}\frac{\beta_K}{2}(\|z^{(0)}-z^\star\|^2_2+1).
\end{equation}

Using Lemma \ref{gradbound} and \cite[Lemma 3]{nesterov2009primal}, $\delta_{K+1}^{-1}\frac{\beta_K}{2}$ can be 
bounded as  follows. 
\begin{equation}
\begin{array}{rl}
\delta_{K+1}^{-1}\frac{\beta_K}{2}\le&\frac{1}{2}\left(\sum_{k=0}^K\frac{1}{\|G(z^{(k)})\|_2}\right)^{-1}\left(\frac{1}{1+\sqrt{3}}+\sqrt{2k+1}\right)\\[5pt]
\le&\frac{1}{2}(\sum_{k=0}^K\frac{1}{C})^{-1}\left(\frac{1}{1+\sqrt{3}}+\sqrt{2k+1}\right)\le\frac{C}{2(K+1)}\left(\frac{1}{1+\sqrt{3}}+\sqrt{2k+1}\right).
\end{array}
\end{equation}
\end{proof}

\subsection{Primal-dual subgradient method}
\label{app:subgrad.comp}
\paragraph{Subgradient computing.}
Given $C$ as a subset of $\R^n$, recall that $\conv(C)$ stands for the convex hull  generated by $C$. 

The subdifferential of $F_i$ is computed by the formula:
\begin{equation}
    \partial F_i(x)=\begin{cases}\partial f_i(x)&\text{ if }f_i(x)>0\,,\\
    \conv(\partial f_i(x)\cup \{0\})&\text{ if }f_i(x)=0\,,\\
    \{0\}&\text{ otherwise}.
    \end{cases}
\end{equation}
Since the subdifferential of $\|\cdot\|^s_2$ function is computed by the formula
\begin{equation}
    \partial \|\cdot\|_2^s(z)=\begin{cases}
    \{2{z}\}&\text{ if }s= 2\,,\\
    s\|z\|_2^{s-2}{z}&\text{ if }z\ne 0\text{ and }s\in[0,2)\,,\\
    \{sg\in\R^m: \|g\|_2\le 1\}&\text{ if }z= 0\text{ and }s\in[0,2),
    \end{cases}
\end{equation}
we have
\begin{equation}\label{eq:differe.composition}
\begin{array}{rl}
    \partial \|F(\cdot)\|_2^s(x)&\supset \conv\{\sum_{i=1}^ma_i b_i\,:\,a\in \partial \|\cdot\|_2^s(F(x))\,,\,b_i\in \partial F_i(x)\}\\[5pt]
    &\supset \begin{cases} {s\|F(x)\|_2^{s-2}\sum_{i=1}^mF_i(x) \partial F_i(x) }&\text{ if }F(x)\ne 0\,,\\
    \{0\}&\text{ otherwise}\,,
    \end{cases}
    \end{array}
\end{equation}
and
\begin{equation}\label{eq:differe.composition.affine}
\begin{array}{rl}
    \partial \|A\cdot-b\|_2^s(x)&\supset \conv\{A^\top u\,:\,u\in \partial \|\cdot\|_2^s(Ax-b)\}\\
    &\supset \begin{cases} {s\|Ax-b\|_2^{s-2}A^\top(Ax-b)}&\text{ if }Ax-b\ne 0\,,\\
    \{0\}&\text{ otherwise}\,.
    \end{cases}
    \end{array}
\end{equation}
 
\paragraph{Proof of Theorem \ref{theo:convergence}.}
In \cite[Section 8]{boyd2014subgradient}, author uses a wrong statement that if $(a_k)_{k=1}^\infty$ and $(b_k)_{k=1}^\infty$ are two nonnegative real sequence such that $\sum_{k=1}^\infty a_kb_k$ is bounded and $\sum_{k=1}^\infty a_k$ diverges, then $b_k\to 0$ as $k\to \infty$. 
For counterexample, take $a_k=1/k$ and
\begin{equation}
    b_k=\begin{cases} 
1&\text{if }k=i^2 \text{ for some }i\in\N\,,\\
0&\text{otherwise.}
\end{cases}
\end{equation}
It is not hard to prove Theorem \ref{theo:convergence} relying on the following lemma:
\begin{lemma}\label{lem:theo:convergence}
Let $K=\infty$ in method \eqref{alg:PDS}.
For every $k\in\N$, let $i_k$ be any element belonging to 
\begin{equation}
    \argmin_{1\le i\le k} T^{(i)\top} (z^{(i)}-z^\star)\,.
\end{equation}
Then 
\begin{equation}\label{eq:conver.subseq}
         |f_0(x^{(i_k)})-p^\star|\to 0\quad\text{ and }\quad
          \|F(x^{(i_k)})\|_2+
          \|Ax^{(i_k)}-b\|_2\to 0
\end{equation}
as $k\to \infty$ with the rate at least $\mathcal O({(k+1)^{\delta/(2s)}})$.
\end{lemma}

\begin{proof}

Let $z^\star:=(x^\star,\lambda^\star,\nu^\star)$. 
Let $R>0$ be large enough such that $R\ge \|z^{(1)}\|_2$ and $R\ge \|z^\star\|_2$.
We start by writing out a basic identity
\begin{equation}
\begin{array}{rl}
    \|z^{(k+1)}-z^\star\|_2^2&=\|z^{(k)}-z^\star\|_2^2 -2\alpha_k T^{(k)\top}(z^{(k)}-z^\star)+\alpha_k^2\|T^{(k)}\|_2^2\\
    &=\|z^{(k)}-z^\star\|_2^2 -2\gamma_k\frac{ T^{(k)\top}}{\|T^{(k)}\|_2}(z^{(k)}-z^\star)+\gamma_k^2\,.
    \end{array}
\end{equation}
By summing it over $k$ and rearranging the terms, we get
\begin{equation}\label{eq:bound.norm}
    \|z^{(k+1)}-z^\star\|_2^2+2\sum_{i=1}^k\gamma_i\frac{ T^{(i)\top}}{\|T^{(i)}\|_2}(z^{(i)}-z^\star) = \|z^{(1)}-z^\star\|^2_2 +\sum_{i=1}^k\gamma_i^2\le 4R^2+S.
\end{equation}
The latter inequality is due to $\sum_{i=1}^\infty\gamma_i^2=S<\infty$
(see Abel's summation formula).

We argue that the sum on the lefthand side is nonnegative.
 First, we estimate a lower bound of 
 \begin{equation}
 \begin{array}{rl}
     T^{(k)\top}(z^{(k)}-z^\star)=&\partial_x L_{\rho}(z^{(k)})^\top(x^{(k)}-x^\star)\\
     &-\partial_\lambda L_{\rho}(z^{(k)})^\top(\lambda^{(k)}-\lambda^\star)
     -\partial_\nu L_{\rho}(z^{(k)})^\top(\nu^{(k)}-\nu^\star)\,.
     \end{array}
 \end{equation}
With $F^{(k)}\ne0$ and $Ax^{(k)}\ne b$, the first term further expands to
\begin{equation}
\begin{array}{rl}
    &\partial_x L_{\rho}(z^{(k)})^\top(x^{(k)}-x^\star) \\[5pt]
    =&  g_0^{(k)\top}(x^{(k)}-x^\star)\\[5pt]
    &+\sum\limits_{i=1}^m\left(\lambda_i^{(k)}+  \rho s\| F^{(k)}\|_2^{s-2} F_i^{(k)}\right) g_i^{(k)\top}(x^{(k)}-x^\star) \\[10pt]
    &+  \nu^{(k)\top}A(x^{(k)}-x^\star)+ \rho s \|Ax^{(k)}-b\|_2^{s-2}(Ax^{(k)}-b)^\top A(x^{(k)}-x^\star)\\[5pt]
     \ge & f_0(x^{(k)}) - p^\star + \lambda^{(k)\top} F^{(k)} + \rho s\| F^{(k)}\|_2^s\\[5pt]
     &+\nu^{(k)\top}(Ax^{(k)}-b) +\rho s\|Ax^{(k)}-b\|_2^s\,. \qquad\text{(since $Ax^\star=b$)}
\end{array}
\end{equation}
It is due to definition of subgradient, for the objective function, we have 
\begin{equation}
    g_0^{(k)\top}(x^{(k)}-x^\star)\ge f_0(x^{(k)}) - p^\star\,,
\end{equation}
and for the constraints
\begin{equation}
    g_i^{(k)\top}(x^{(k)}-x^\star)\ge F_i^{(k)} - F_i(x^\star)=F_i^{(k)}\,,\,i=1,\dots,m\,.
\end{equation}
Notice that $\lambda_i^{(k)}+ \rho s\| F^{(k)}\|_2^{s-2} F_i^{(k)}$ is nonnegative since $\lambda_i^{(k)}$ and $F_i^{(k)}$ are nonnegative.
Next, we have
\begin{equation}
    -\partial_\lambda L_{\rho}(z^{(k)})^\top(\lambda^{(k)}-\lambda^\star)=-F^{(k)\top} (\lambda^{(k)}-\lambda^\star)
\end{equation}
 and 
 \begin{equation}
      -\partial_\nu L_{\rho}(z^{(k)})^\top(\nu^{(k)}-\nu^\star)=(b-Ax^{(k)})^\top(\nu^{(k)}-\nu^\star)\,.
 \end{equation}
Using these and subtracting, we obtain
\begin{equation}\label{eq:lower.bound.T}
\begin{array}{rl}
    T^{(k)\top}(z^{(k)}-z^\star)\ge& f_0(x^{(k)}) - p^\star  +\lambda^{\star\top}F^{(k)}+ \nu^{\star\top}(Ax^{(k)}-b)\\[5pt]
    &+\rho s(\| F^{(k)}\|_2^s + \|Ax^{(k)}-b\|_2^s)\\[5pt]
    =& L(x^{(k)},\lambda^\star,\nu^\star)-L(x^\star,\lambda^\star,\nu^\star)\\[5pt]
    &+\rho s(\| F^{(k)}\|_2^s + \|Ax^{(k)}-b\|_2^s)\\
    \ge & 0\,.
\end{array}
\end{equation}
The latter inequality is implied by using \eqref{eq:saddle.point}.
It is remarkable that \eqref{eq:lower.bound.T} is still true even if $F^{(k)}=0$ or $Ax^{(k)}=b$.

Since both terms on the lefthand side of \eqref{eq:bound.norm} are nonnegative, for all $k$, we have
\begin{equation}\label{eq:estimate}
    \|z^{(k+1)}-z^\star\|_2^2\le 4R^2+S\quad \text{and}\quad 2\sum_{i=1}^k\gamma_i\frac{ T^{(i)\top}}{\|T^{(i)}\|_2}(z^{(i)}-z^\star)\le 4R^2+S\,.
\end{equation}
The first inequality yields that there exists positive real $D$ satisfying $\|z^{(k)}\|\le D$, namely $D=R+\sqrt{4R^2+S}$. 
By assumption, the norm of subgradients $g_i^{(k)}$, $i=0,1,\dots,m$ on the set $\|x^{(k)}\|_2$ is bounded, so it follows that $\|T^{(k)}\|_2$ is bounded by some positive real $C$ independent from $k$.

The second inequality of \eqref{eq:estimate} implies that 
\begin{equation}
         T^{(i_k)\top}(z^{(i_k)}-z^\star)\le \frac{C(4R^2+S)}{2\sum_{i=1}^k\gamma_i}\le \frac{W}{(k+1)^{\delta/2}} \,.
\end{equation}
for some $W>0$ independent from $k$. 
The latter inequality is due to the asymptotic behavior of the zeta function.
Moreover, \eqref{eq:lower.bound.T} turns out that
\begin{equation}
\begin{array}{rl}
    & (L(x^{(i_k)},\lambda^\star,\nu^\star)-L(x^\star,\lambda^\star,\nu^\star))+\rho s \| F(x^{(i_k)})\|_2^s + \rho s\|Ax^{(i_k)}-b\|_2^s  \\[5pt]
    \le &T^{(i_k)\top}(z^{(i_k)}-z^\star)\le \frac{W}{(k+1)^{\delta/2}}\,.
\end{array}
\end{equation}
Since three terms on the lefthand size are nonnegative, we obtain 
\begin{equation}
     \| F(x^{(i_k)})\|_2\le \frac{W^{1/s}}{(\rho s)^{1/s}(k+1)^{\delta/(2s)}}\,,\qquad \|Ax^{(i_k)}-b\|_2\le \frac{W^{1/s}}{(\rho s)^{1/s}(k+1)^{\delta/(2s)}}\,,
\end{equation}
and
\begin{equation}
\begin{array}{rl}
    \frac{W}{(k+1)^{\delta/2}}\ge& L(x^{(i_k)},\lambda^\star,\nu^\star)-L(x^\star,\lambda^\star,\nu^\star)\\
   = &f_0(x^{(i_k)}) - p^\star  +\lambda^{\star\top}F(x^{(i_k)})+ \nu^{\star\top}(Ax^{(i_k)}-b)\ge 0\,.
    \end{array}
\end{equation}
Using these, we have 
\begin{equation}
    \frac{W}{(k+1)^{\delta/2}}+\xi_{i_k}\ge f_0(x^{(i_k)}) - p^\star  \ge -\xi_{i_k}\,.
\end{equation}
where $\xi_{i_k}=\|\lambda^{\star}\|_2\|F(x^{(i_k)})\|_2+\|\nu^{\star}\|_2\|Ax^{(i_k)} - b\|_2$.
Since $\xi_{i_k}\to 0$ as $k\to \infty$ with the rate at least $\mathcal{O}((k+1)^{\delta/(2s)})$. 
This proves \eqref{eq:conver.subseq}.
\end{proof}

\bibliographystyle{abbrv}

\end{document}